\documentclass{amsart}
\usepackage[utf8]{inputenc}

\title{Quantification of Banach-Saks properties of higher orders}
\author[Z. Silber]{Zdeněk Silber}
\email{zdesil@seznam.cz}
\keywords{Banach-Saks set, Weak Banach-Saks set, Quantitative version, $\ell_1^\xi$-spreading model, Schreier families}
\address{Department of Mathematical Analysis, Faculty of Mathematics and Physics, Charles University, Sokolovská 83, 186 75, Praha 8, Czech Republic}
\thanks{The research was supported by the Charles University, project GAUK no. 154121 and the grant SVV-2020-260583.}
\subjclass[2010]{46B20; 46B50}

\usepackage{amsthm}
\usepackage{amsmath}
\usepackage{mathrsfs}
\usepackage{mathtools}
\usepackage[shortlabels]{enumitem}
\usepackage{verbatim}
\usepackage{amsfonts}
\usepackage{amssymb}
\usepackage{graphicx}
\usepackage[a-2u]{pdfx}
\usepackage{tikz}
\usepackage[most]{tcolorbox}
\usepackage{esint}
\usepackage{enumitem}
\usetikzlibrary{matrix,arrows}
\graphicspath{{images/}}

\newtheorem{theorem}{Theorem}[section]
\newtheorem{lemma}[theorem]{Lemma}
\newtheorem{prop}[theorem]{Proposition}

\newtheorem{question}{Question}

\theoremstyle{definition}
\newtheorem*{definition}{Definition}
\newtheorem*{remark}{Remark}
\newtheorem{example}[theorem]{Example}

\newcommand{\norm}[1]{\left\lVert#1\right\rVert}

\begin{document}

\begin{abstract}
    We investigate possible quantifications of Banach-Saks sets and weak Banach-Saks sets of higher orders and their relations to other quantities. We prove a quantitative version of the characterization of weak $\xi$-Banach-Saks sets using $\ell_1^{\xi+1}$-spreading models and a quantitative version of the relation of $\xi$-Banach-Saks sets, weak $\xi$-Banach-Saks sets, norm compactness and weak compactness. We further introduce a new measure of weak compactness. Finally, we provide some examples showing the limitations of these quantifications.
\end{abstract}

\maketitle

\section{Introduction}\label{section:Intro}

A Banach space $X$ is said to have the \textit{Banach-Saks property} if every bounded sequence in $X$ admits a Ces\`aro summable subsequence. This property was first investigated by Banach and Saks in \cite{BanachSaks1930}, where they showed that the spaces $L_p$ for $1<p<\infty$ enjoy this property. Every space with the Banach-Saks property is reflexive \cite{NishiuraWaterman1963} but there are reflexive spaces which do not have the Banach-Saks property, see \cite{Baernstein1972} or Example \ref{ExampleXiSchreierStar} below. However, every uniformly convex space (or more generally every super-reflexive space, as super-reflexive spaces admit a uniformly convex renorming \cite{Enflo1972}) has the Banach-Saks property \cite{Kakutani1939}.

A Banach space $X$ has the \textit{weak Banach-Saks property} if every weakly convergent sequence in $X$  admits a Ces\`aro summable subsequence. For reflexive spaces the weak Banach-Saks property and Banach-Saks property are equivalent but there are non-reflexive spaces that have the weak Banach-Saks property, like $c_0$ or $L_1$, see \cite{Farnum1974} and \cite{Szlenk1965}.

There is a localized version of these properties -- a bounded set $A$ in a Banach space $X$ is said to be a \textit{Banach-Saks set}, if every sequence in $A$ admits a Ces\`aro summable subsequence, and is called a \textit{weak Banach-Saks set}, if every weakly convergent sequence in $A$ admits a Ces\`aro summable subsequence. It follows that a Banach space $X$ has the Banach-Saks property, resp. the weak Banach-Saks property, if and only if its closed unit ball $B_X$ is a Banach-Saks set, resp. a weak Banach-Saks set. It is easy to see that a relatively weakly compact weak Banach-Saks set is a Banach-Saks set. The other implication also holds. Indeed, a Banach-Saks set is obviously a weak Banach-Saks set and the fact that it is also relatively weakly compact follows from \cite[Proposition 2.3.]{LopezRuizTradacete2014}. A quantitative version of this statement was investigated in \cite{Bendov__2015}.

The property of being a weak Banach-Saks set is closely tied to two other notions, which will be explained in detail in Section \ref{section:Prep} below. First of them is the notion of an \textit{$\ell_1$-spreading model}. Recall that a bounded sequence $(x_n)_{n \in \mathbb{N}}$ in a Banach space $X$ is said to generate an $\ell_1$-spreading model if there is a positive constant $c$ such that for all finite subsets $F$ of $\mathbb{N}$ satisfying $|F| \leq \min F$, where $|F|$ is the cardinality of the set $F$, and all sequences $(a_i)_{i \in F}$ of scalars we have
\begin{align*}
    \norm{ \sum_{i \in F} a_i x_i} \geq c \sum_{i \in F} |a_i|.
\end{align*}
The second related notion is \textit{uniform weak convergence}. A sequence $(x_n)_{n \in \mathbb{N}}$ is uniformly weakly convergent to $x$ if for each $\epsilon > 0$ there exists $n \in \mathbb{N}$ such that for all $x^* \in B_{X^*}$ we have
\begin{align*}
    \# \{k \in \mathbb{N}: \; |x^*(x_k-x)| \geq \epsilon\} \leq n,
\end{align*}
where $\# A$ is another notation for the cardinality of a set $A$. It follows from \cite[Section 2]{LopezRuizTradacete2014} that a bounded set $A$ in a Banach space $X$ is a weak Banach-Saks set, if and only if no weakly convergent sequence in $A$ generates an $\ell_1$-spreading model, if and only if every weakly convergent sequence in $A$ admits a uniformly weakly convergent subsequence. Quantitative version of this result was also provided in \cite{Bendov__2015}.

It follows from the Mazur's theorem that if we have a weakly null sequence $(x_n)_{n \in \mathbb{N}}$ in a Banach space $X$, then there is a sequence of convex combinations which converges to zero in norm. The weak Banach-Saks property of $X$ then means that these convex combinations can be chosen to be the Ces\`aro sums. In \cite{ArgyrosMercourakusTsarpalias} the authors investigated how regular these convex combinations can be in spaces failing the weak Banach-Saks property and defined the \textit{$\xi$-Banach-Saks property} and the \textit{weak $\xi$-Banach-Saks property} for a countable ordinal $\xi$ (see Section \ref{section:Prep}). The main goal of this paper is to provide quantifications, analogous to those provided in \cite{Bendov__2015}, for the properties of higher orders. The investigation of these properties also led to a new measure of weak non-compactness.

\section{Preparation}\label{section:Prep}

\subsection{Notation}

For an infinite subset $M$ of $\mathbb{N}$ we will denote by $[M]$ the set of all infinite subsets of $M$. On the other hand, for any subset $M$ of $\mathbb{N}$ we will denote by $[M]^{<\infty}$ the set of all finite subsets of $M$. If $n \in \mathbb{N}$ we will denote by $[M]^{<n}$ the sets of all subsets of $M$ of cardinality less than $n$.

If $M$ is an infinite subset of $\mathbb{N}$ and we write $M = (m_n)_{n \in \mathbb{N}}$, then we always mean that $M = \{m_n: \; n \in \mathbb{N}\}$ and $m_1 < m_2 < \dots$. We also use an analogous convention for finite subsets of $\mathbb{N}$.

For a Banach space $X$ we will denote by $B_X$ the closed unit ball of $X$ and by $S_X$ the unit sphere of $X$. In the special case where $X = \ell_1$, we will denote by $S_{\ell_1}^+$ the set of those elements of $S_{\ell_1}$ which have non-negative coordinates.

If $(a_n)_{n \in \mathbb{N}} \in \ell_1$, $F$ is a subset of integers and $(x_n)_{n \in \mathbb{N}}$ is a bounded sequence in a Banach space $X$, we set
\begin{itemize}
    \item $\langle (a_n)_{n \in \mathbb{N}}, F \rangle = \underset{n \in F}{\sum} a_n$;
    \item $(a_n)_{n \in \mathbb{N}} \cdot (x_n)_{n \in \mathbb{N}} = \underset{n \in \mathbb{N}}{\sum} a_n x_n$.
\end{itemize}
We denote the canonical basis of the space $c_{00}$ of eventually zero sequences by $(e_n)_{n \in \mathbb{N}}$.

Let $A, B$ be subsets of $\mathbb{N}$. If we write $A < B$, then we mean that $\max A < \min B$. Analogously, $A \leq B$ means that $\max A \leq \min B$. We write $n \leq A$, resp. $n < A$, instead of $\{n\} \leq A$, resp. $\{n\} < A$.

If $F$ is a finite set, we will write $|F|$ or $\# F$ for the cardinality of $F$.

If $(x_n)_{n \in \mathbb{N}}$ is a sequence and $M = (m_n)_{n \in \mathbb{N}} \in [\mathbb{N}]$, then we denote the subsequence $(x_{m_n})_{n \in \mathbb{N}}$ by $(x_n)_{n \in M}$.

\subsection{Families of subsets of integers}

We identify subsets of $\mathbb{N}$ with their characteristic functions, and thus with elements of the Cantor set $\{0,1\}^{\mathbb{N}}$. This characterization provides us with a metrizable topology on the power set of $\mathbb{N}$.

\begin{definition}
    Let $\mathcal{F}$ be a family of finite sets of integers. We say that $\mathcal{F}$ is
    \begin{itemize}
        \item \textit{Hereditary}, if $A \in \mathcal{F}$ and $B \subseteq A$ implies $B \in F$;
        \item \textit{Precompact}, if the closure of $\mathcal{F}$ consists only of finite sets;
        \item \textit{Adequate}, if it is both hereditary and precompact.
    \end{itemize}
    If $M \in [\mathbb{N}]$, we define the \textit{trace} of $\mathcal{F}$ on $M$ by
    \begin{align*}
        \mathcal{F}[M] = \{F \cap M: \; F \in \mathcal{F}\}.
    \end{align*}
\end{definition}

Note that the trace of an adequate family is also adequate. If $\mathcal{F}$ is hereditary, then $\mathcal{F}[M] = \{F \in \mathcal{F}: \; F \subseteq M\}$.

\subsection{Schreier families and Repeated Averages} \label{subsection:SHRA}

In this subsection we will define the Schreier families and the Repeated Averages. For a countable limit ordinal $\xi$ we fix an increasing sequence $(\xi_n)_{n \in \mathbb{N}}$ of successor ordinals with $\xi = \sup \xi_n$. This choice is necessary for us to define the Schreier families and Repeated Averages for limit ordinals. While these definitions certainly depend on this choice, some of the quantities defined in the following subsection do not. The independence on this choice will be discussed in detail in Section \ref{section:Remarks} below.

\begin{definition}
    The \textit{Schreier families} $(\mathcal{S}_\xi)_{\xi < \omega_1}$ are defined recursively. First we define the family $\mathcal{S}_0$ as
    \begin{align*}
        \mathcal{S}_0 = \{\{n\}: n \in \mathbb{N}\} \cup \{\emptyset\}.
    \end{align*}
    For a successor ordinal $\xi + 1 < \omega_1$ we define
    \begin{align*}
        \mathcal{S}_{\xi+1} = \left\{\bigcup_{i=1}^n F_i: \; n \leq F_1 < F_2 < \cdots < F_n, \; F_i \in \mathcal{S}_\xi, \; n \in \mathbb{N} \right\} \cup \{\emptyset\}
    \end{align*}
    and for a limit ordinal $\xi < \omega_1$ we take the fixed increasing sequence of successor ordinals $(\xi_n)_{n \in \mathbb{N}}$ with $\xi = \sup \xi_n$ and define
    \begin{align*}
        \mathcal{S}_\xi = \left\{F \in \mathcal{S}_{\xi_n}: \; n \leq F, \; n \in \mathbb{N} \right\} \cup \{\emptyset\}.
    \end{align*}
\end{definition}

Note that the family $\mathcal{S}_{\xi+1}$ always contains the family $\mathcal{S}_\xi$. On the other hand, it is not generally true that the family $\mathcal{S}_\zeta$ contains the family $\mathcal{S}_\xi$ for $\zeta > \xi$. It does, however, contain all the sets from $\mathcal{S}_\xi$ with sufficiently large minimal element, see \cite[Lemma 2.1.8.(a)]{ArgyrosMercourakusTsarpalias}. The family $\mathcal{S}_1$ is the classical Schreier family
\begin{align*}
    \mathcal{S}_1 = \{F \in [\mathbb{N}]^{< \infty}: \; |F| \leq \min F\} \cup \{\emptyset\}.
\end{align*}
It is readily proved by induction that the families $\mathcal{S}_\xi$, $\xi < \omega_1$, are adequate and have the following spreading property:
\begin{align*}
    &\text{If } F = (f_1,\dots,f_n) \in \mathcal{S}_\xi \text{ and } G = (g_1,\dots,g_n) \text{ is such} \\
    &\text{that } f_i \leq g_i, \; i = 1,\dots,n, \text{ then } G \in \mathcal{S}_\xi.
\end{align*}

\begin{definition}
    Let $\xi < \omega_1$ and $M = (m_n)_{n \in \mathbb{N}} \in [\mathbb{N}]$. We define
    \begin{align*}
        \mathcal{S}_\xi^M = \{(m_i)_{i \in F}: \; F \in \mathcal{S}_\xi\}.
    \end{align*}
\end{definition}

In the case that $\xi = 0$, we have that
\begin{align*}
    \mathcal{S}_0^M = \mathcal{S}_0 [M] = \{ \{m_n\}: \; n \in \mathbb{N} \} \cup \{\emptyset\}.
\end{align*}
However, if $\xi > 0$, then $\mathcal{S}_\xi^M \subsetneq \mathcal{S}_\xi [M]$. Indeed $\mathcal{S}_\xi^M$ is a subset of $\mathcal{S}_\xi$ by the spreading property of the family $\mathcal{S}_\xi$ as $i \leq m_i$ for each $i \in \mathbb{N}$, and the sets from $\mathcal{S}_\xi^M$ are obviously subsets of $M$. The fact that the inclusion is strict can be proved by induction and is illustrated by the following example: If we set $m_n = n + 1$ and $M = (m_n)_{n \in \mathbb{N}} \in [\mathbb{N}]$, then the set $\{2,3\} \in \mathcal{S}_1[M] \setminus \mathcal{S}_1^M$. For more information about the relation of the families $\mathcal{S}_\xi^M$ and $\mathcal{S}_\xi [M]$ see \cite[Remark 2.1.12]{ArgyrosMercourakusTsarpalias}.

\begin{definition}
    Let $M  \in [\mathbb{N}]$. An $M$-summability method is a sequence $(A_n)_{n \in \mathbb{N}}$ where $A_n \in S_{\ell_1}^+$ are such that $\operatorname{supp}A_n < \operatorname{supp} A_{n+1}$ for all $n \in \mathbb{N}$ and $M = \bigcup_{n=1}^\infty \operatorname{supp} A_n$, where $\operatorname{supp} F$ denotes the support of an element $F$ of $\ell_1$, that is the set of coordinates where $F$ is nonzero.
    
    We say that a bounded sequence $(x_n)_{n \in \mathbb{N}}$ in some Banach space $X$ is $(A_n)_{n \in \mathbb{N}}$-summable if the sequence $\left( A_n \cdot (x_k)_{k \in \mathbb{N}} \right)_{n \in \mathbb{N}}$ is Ces\`aro summable.
\end{definition}

Note that if $A_n = e_{m_n}$ for some increasing sequence $M = (m_n)_{n \in \mathbb{N}}$ of integers, then the $(A_n)_{n \in \mathbb{N}}$-summability of a sequence $(x_n)_{n \in \mathbb{N}}$ is just the Ces\`aro summability of the subsequence $(x_{m_n})_{n \in \mathbb{N}}$. One important fact we will need later is the simple observation that the summability methods preserve convergence. We will specifically use that if a sequence $(x_n)_{n \in \mathbb{N}}$ is weakly null, then $(A_n \cdot (x_k)_{k \in \mathbb{N}})_{n \in \mathbb{N}}$ is also weakly null for any $M$-summability method $(A_n)_{n \in \mathbb{N}}$.

The Repeated Averages are a special type of summability methods that arise by iterating consecutive averages.

\begin{definition}
    Let $M = (m_n)_{n \in \mathbb{N}} \in [\mathbb{N}]$. The \textit{Repeated Averages} are the $M$-summability methods $(\xi_n^M)_{n \in \mathbb{N}}$, $\xi < \omega_1$, which are defined recursively in the following way.
    \begin{enumerate}
        \item If $\xi = 0$, we set $\xi_n^M = e_{m_n}$, $n \in \mathbb{N}$.
        \item If $\xi = \zeta + 1$ and $(\zeta_n^M)_{n \in \mathbb{N}}$ have already been defined, we recursively define $\xi_n^M$ in the following way
        \begin{align*}
            &k_1 = 0, \;s_1 = \min \operatorname{supp} \zeta^M_1 = \min M, & \xi_1^M = \frac{1}{s_1} \sum_{i= 1}^{s_1} \zeta_i^M \\
            & \vdots & \\
            &k_n = k_{n-1} + s_{n-1}, \; s_n = \min \operatorname{supp} \zeta^M_{k_n+1}, & \xi_n^M = \frac{1}{s_n} \sum_{i = k_n + 1}^{k_n + s_n} \zeta_i^M \\
            & \vdots &
        \end{align*}
        \item If $\xi$ is a limit ordinal and $(\zeta^N_n)_{n \in \mathbb{N}}$ have already been defined for all $\zeta < \xi$ and $N \in [\mathbb{N}]$, we take the increasing sequence of successor ordinals $(\xi_n)_{n \in \mathbb{N}}$ which was used to define the Schreier family $\mathcal{S}_\xi$. We use the notation $[\xi_{n}]_j^N$ for the already defined summability method $\zeta_j^N$ for $\zeta = \xi_n$. Set
        \begin{align*}
            &M_1 = M, & n_1 = m_1, \\
            &M_2 = M_1 \setminus \operatorname{supp} [\xi_{n_1}]_1^{M_1}, & n_2 = \min M_2, \\
            &\vdots& \\
            &M_j = M_{j-1} \setminus \operatorname{supp} [\xi_{n_{j-1}}]_1^{M_{j-1}}, & n_j = \min M_j, \\
            &\vdots&
        \end{align*}
        Finally we set for $j \in \mathbb{N}$
        \begin{align*}
            \xi_j^M = [\xi_{n_j}]_1^{M_j}.
        \end{align*}
    \end{enumerate}
\end{definition}

It is readily proved by induction that for each $\xi < \omega_1$ and $M \in [\mathbb{N}]$ the sequence $(\xi^M_n)_{n \in \mathbb{N}}$ is an $M$-summability method. We say that a sequence $(x_n)_{n \in \mathbb{N}}$ is $(\xi,M)$-summable instead of $(\xi^M_n)_{n \in \mathbb{N}}$-summable.

A nice property of the Repeated Averages is that their supports are elements of the corresponding Schreier family, that is $\operatorname{supp} \xi_n^M \in \mathcal{S}_\xi [M]$ for all $M \in [\mathbb{N}]$ and $n \in \mathbb{N}$.

\subsection{$\ell_1^\xi$-spreading models and (weak) $\xi$-Banach-Saks sets}

\begin{definition}
    Let $(x_n)_{n \in \mathbb{N}}$ be a bounded sequence in a Banach space $X$, $\xi < \omega_1$ and $c > 0$. We say that $(x_n)_{n \in \mathbb{N}}$ generates an $\ell_1^\xi$-spreading model with constant $c$ if
    \begin{align*}
        \forall F \in \mathcal{S}_\xi \; \forall (\alpha_i)_{i \in F} \in \mathbb{R}^F : \; \norm{\sum_{i \in F} \alpha_i x_i} \geq c \sum_{i \in F} |\alpha_i|.
    \end{align*}
    We say that $(x_n)_{n \in \mathbb{N}}$ generates an $\ell_1^\xi$-spreading model if it generates an $\ell_1^\xi$-spreading model with some constant $c > 0$.
\end{definition}

This definition generalises the classical notion of an $\ell_1$-spreading model, which corresponds to the case $\xi = 1$.

\begin{definition}
    Let $A$ be a bounded subset of a Banach space $X$ and $\xi < \omega_1$. We say that $A$ is a $\xi$-Banach-Saks set if for every sequence $(x_n)_{n \in \mathbb{N}}$ in $A$ there is some $M \in [\mathbb{N}]$ such that $(x_n)_{n \in \mathbb{N}}$ is $(M,\xi)$-summable.
    
    $A$ is called a weak $\xi$-Banach-Saks set if the same property holds for weakly convergent sequences in $A$, that is, if for every sequence $(x_n)_{n \in \mathbb{N}}$ in $A$ weakly convergent to some $x \in X$ there is some $M \in [\mathbb{N}]$ such that $(x_n)_{n \in \mathbb{N}}$ is $(M,\xi)$-summable.
\end{definition}

These definitions generalise the notion of a Banach-Saks set and a weak Banach-Saks set, which correspond to the case $\xi = 0$. Following \cite{Bendov__2015}, we will now define some quantities that we will later use to quantify the notions of (weak) $\xi$-Banach-Saks sets and $\ell_1^{\xi}$-spreading models.

\begin{definition}
    Let $(x_n)_{n \in \mathbb{N}}$ be a bounded sequence in a Banach space $X$. We define the following two quantities
    \begin{itemize}
        \item $\operatorname{ca}(x_n) = \underset{n \in \mathbb{N}}{\inf} \sup \{\norm{x_k - x_l}: \; k,l \geq n\}$;
        \item $\operatorname{cca}(x_n) = \operatorname{ca} \left( \frac{1}{n} \sum_{i=1}^n x_i \right)$.
    \end{itemize}
\end{definition}

The quantity $\operatorname{ca}$ measures how far a given sequence is from being norm Cauchy. Indeed, $\operatorname{ca}(x_n) = 0$ if and only if the sequence $(x_n)_{n \in \mathbb{N}}$ is norm Cauchy. The quantity $\operatorname{cca}$ then measures how far are the Ces\`aro sums of a given sequence from being norm Cauchy.

\begin{definition}
    Let $(x_n)_{n \in \mathbb{N}}$ be a bounded sequence in a Banach space $X$ and $\xi < \omega_1$. We define
    \begin{align*}
        \widetilde{\operatorname{cca}}_\xi ((x_n)_{n \in \mathbb{N}}) &= \inf_{M \in [\mathbb{N}]} \left( \operatorname{cca}(\xi^M_n \cdot (x_k)_{k \in \mathbb{N}}) \right) \\
        \widetilde{\operatorname{cca}}_\xi^s ((x_n)_{n \in \mathbb{N}}) &= \sup_{M \in [\mathbb{N}]} \left( \inf_{N \in [M]} \operatorname{cca}(\xi^N_n \cdot (x_k)_{k \in \mathbb{N}}) \right).
    \end{align*}
\end{definition}
The quantity $\widetilde{\operatorname{cca}}_0$ is the quantity $\widetilde{\operatorname{cca}}$ used in \cite{Bendov__2015} and measures how far a given sequence is from containing a Ces\`aro summable subsequence. We will, however, mostly work with the quantity $\widetilde{\operatorname{cca}}_0^s$, which measures if all subsequences of a given sequence contain a further subsequence which is Ces\`aro summable, and with its generalizations for $\xi > 0$. The precise correspondence between these quantities for $\xi = 0$ is
\begin{align*}
    \widetilde{\operatorname{cca}}_0^s ((x_n)_{n \in \mathbb{N}}) = \sup \{\widetilde{\operatorname{cca}}_0 ((y_n)_{n \in \mathbb{N}}): \; (y_n)_{n \in \mathbb{N}} \text{ is a subsequence of } (x_n)_{n \in \mathbb{N}}\},
\end{align*}
for larger $\xi$ the correspondence is not so clear.

Now we can define the quantifications of the notions of (weak) $\xi$-Banach-Saks sets and $\ell_1^\xi$-spreading models.

\begin{definition}
    Let $A$ be a bounded subset of a Banach space $X$ and $\xi < \omega_1$. We define the following quantities:
    \begin{align*}
        \operatorname{sm}_\xi (A) = \sup \{ c > 0: \;& \text{there is a sequence } (x_n)_{n \in \mathbb{N}} \text{ in } A \text{ weakly convergent} \\
        &\text{to some } x \in X \text{ such that } (x_n - x)_{n \in \mathbb{N}} \text{ generates} \\
        &\text{an } \ell_1^\xi \text{-spreading model with constant } c \},
    \end{align*}
    where we set the supremum of the empty set to be zero, and
    
    \begin{align*}
        \operatorname{bs}_\xi (A) &= \sup \{\widetilde{\operatorname{cca}}_\xi (x_n): \; (x_n)_{n \in \mathbb{N}} \text{ is a sequence in } A \} \\
        \operatorname{wbs}_\xi (A) &= \sup \{\widetilde{\operatorname{cca}}_\xi (x_n): \; (x_n)_{n \in \mathbb{N}} \text{ is a weakly convergent sequence in } A \} \\
        \operatorname{bs}_\xi^s (A) &= \sup \{\widetilde{\operatorname{cca}}_\xi^s (x_n): \; (x_n)_{n \in \mathbb{N}} \text{ is a sequence in } A \} \\
        \operatorname{wbs}_\xi^s (A) &= \sup \{\widetilde{\operatorname{cca}}_\xi^s (x_n): \; (x_n)_{n \in \mathbb{N}} \text{ is a weakly convergent sequence in } A \}.
    \end{align*}
\end{definition}

It follows from the definition that $\operatorname{sm}_\xi(A) = 0$ if and only if $A$ contains no sequence $(x_n)_{n \in \mathbb{N}}$ weakly convergent to some $x \in X$ such that $(x_n-x)_{n \in \mathbb{N}}$ generates an $\ell_1^\xi$-spreading model. The fact that $\operatorname{bs}_\xi (A) = 0$ (resp. $\operatorname{wbs}_\xi (A) = 0$) if and only if $A$ is a $\xi$-Banach-Saks set (resp. weak $\xi$-Banach-Saks set) will be shown later in Proposition \ref{PropBS} for $\xi$-Banach-Saks sets and Propostion \ref{PropWBS} for weak $\xi$-Banach-Saks sets. For $\xi = 0$ we have $\operatorname{bs}_0(A) = \operatorname{bs}_0^s(A)$ and $\operatorname{wbs}_0(A) = \operatorname{wbs}_0^s(A)$ as any subsequence of a sequence in $A$ is also a sequence in $A$. For larger $\xi$ we trivially have $\operatorname{bs}_\xi(A) \leq \operatorname{bs}_\xi^s(A)$ and $\operatorname{wbs}_\xi(A) \leq \operatorname{wbs}_\xi^s(A)$. We will show later in Theorem \ref{TheoremQuantBS} that the quantities $\operatorname{wbs}_\xi$ and $\operatorname{wbs}_\xi^s$ are equivalent for $\xi > 0$.

\subsection{$(\xi,c)$-large sets and uniformly weakly converging sequences}

\begin{definition}
    Let $(x_n)_{n \in \mathbb{N}}$ be a weakly null sequence in a Banach space $X$ and $\delta > 0$. We define the family
    \begin{align*}
        \mathcal{F}_\delta ((x_n)_{n \in \mathbb{N}}) = \{F \in [\mathbb{N}]^{< \infty}: \; \text{there is } x^* \in B_{X^*} \text{ with } x^*(x_n) \geq \delta, \; n \in F\}.
    \end{align*}
    We will usually write only $\mathcal{F}_\delta$ instead of $\mathcal{F}_\delta ((x_n)_{n \in \mathbb{N}})$ if it causes no confusion.
\end{definition}

Note that the family $\mathcal{F}_\delta$ is obviously hereditary and is also precompact as the sequence $(x_n)_{n \in \mathbb{N}}$ is weakly null. Indeed, suppose there is a sequence $(F_n)_{n=1}^\infty$ of sets from the family $\mathcal{F}_\delta$ that converges to an infinite set $F \in [\mathbb{N}]$. For $n \in \mathbb{N}$ let $x_n^*$ be an element of $B_{X^*}$ witnessing that $F_n$ belongs to $\mathcal{F}_\delta$ and let $x^*$ be any weak$^*$ cluster point of $(x_n^*)_{n=1}^\infty$. Then, for a fixed $k \in F$, we have $x_n^*(x_k) \geq \delta$ for all but finitely many $n \in \mathbb{N}$. Hence, $x^*(x_k) \geq \delta > 0$ for infinitely many $k \in \mathbb{N}$ and $(x_n)_{n=1}^\infty$ is not weakly null -- a contradiction. Hence, $\mathcal{F}_\delta$ is an adequate family of subsets of $\mathbb{N}$. 

To use some of the results of \cite{ArgyrosMercourakusTsarpalias}, we need to present an alternative definition of $\mathcal{F}$ using weakly compact subsets of $c_0$. Let us define $D = \{(x^*(x_n))_{n=1}^\infty: \; x^* \in B_{X^*}\}$. Then $D$ is a weakly compact subset of $c_0$. Indeed, $D$ is the image of $B_{X^*}$ under the weak$^*$-to-weak continuous mapping $x^* \mapsto (x^*(x_n))_{n=1}^\infty$. It follows that
\begin{align*}
    \mathcal{F}_\delta = \{ F \in [\mathbb{N}]^{<\infty}: \; \text{there is } f \in D \text{ with } f(n) \geq \delta, \; n \in F\}.
\end{align*}

We now recall the definition of uniformly weakly convergent sequences.

\begin{definition}
    A sequence $(x_n)_{n \in \mathbb{N}}$ in a Banach space $X$ is said to be uniformly weakly convergent to some $x$ in $X$ if for each $\epsilon > 0$
    \begin{align*}
        \exists n \in \mathbb{N} \; \forall x^* \in B_{X^*} : \; \# \{k \in \mathbb{N}: \; |x^*(x_k-x)| \geq \epsilon\} \leq n.
    \end{align*}
    Note that the absolute value in this definition can be omitted, that is $(x_n)_{n \in \mathbb{N}}$ is uniformly weakly convergent to $x \in X$ if and only if for all $\epsilon > 0$
    \begin{align*}
        \exists n \in \mathbb{N} \; \forall x^* \in B_{X^*} : \; \# \{k \in \mathbb{N}: \; x^*(x_k-x) \geq \epsilon\} \leq n,
    \end{align*}    
\end{definition}

Uniform weak convergence can be used to characterize Banach-Saks (resp. weak Banach-Saks) sets. Precisely, a bounded set $A$ in a Banach space $X$ is a Banach-Saks (resp. weak Banach-Saks) set, if and only if every (resp. every weakly convergent) sequence in $A$ has a uniformly weakly convergent subsequence, see \cite[Theorem 2.4.]{LopezRuizTradacete2014}.

\begin{definition}
    Let $(x_n)_{n \in \mathbb{N}}$ be a sequence in a Banach space $X$ weakly converging to $x \in X$, $c > 0$ and $\xi < \omega_1$. We say that $(x_n)_{n \in \mathbb{N}}$ is $(\xi,c)$-large if there is $M \in  [\mathbb{N}]$ such that $\mathcal{S}_\xi^M \subseteq \mathcal{F}_c ((x_n - x)_{n \in \mathbb{N}})$.
\end{definition}

\begin{definition}
    Let $A$ be a bounded subset of a Banach space $X$ and $\xi < \omega_1$. We define the quantity
    \begin{align*}
        \operatorname{wus}_\xi (A) = \sup \{c>0: \; &\text{there is a sequence }(x_n)_{n \in \mathbb{N}} \text{ in } A \\
        & \text{weakly convergent to some } x \in X \\
        &\text{such that } (x_n)_{n \in \mathbb{N}} \text{ is } (\xi,c)\text{-large}\},
    \end{align*}
    where we again set the supremum of the empty set to be zero.
\end{definition}

The quantity $\operatorname{wus}_\xi$ is a generalization of the quantity $\operatorname{wus}$ used in \cite{Bendov__2015} which measures how far is $A$ from having the property that every weakly convergent sequence in $A$ has a uniformly weakly convergent subsequence. Indeed, for a bounded set $A$ we have $\operatorname{wus}_1(A) = \operatorname{wus}(A)$ which will follow from the following lemma for sequences. It uses the quantity $\widetilde{\operatorname{wu}}$, which is defined in \cite{Bendov__2015} and used to define the quantity $\operatorname{wus}$.

\begin{lemma}
    Let $(x_n)_{n \in \mathbb{N}}$ be a sequence in a Banach space $X$ weakly convergent to some $x \in X$ and let $c>0$. Then
    \begin{enumerate}[(i)]
        \item If $(x_n)_{n \in \mathbb{N}}$ is $(1,c)$-large, then there is $M \in [\mathbb{N}]$ such that $\widetilde{\operatorname{wu}}((x_n)_{n \in M}) \geq c$.
        \item If $\widetilde{\operatorname{wu}}((x_n)_{n \in \mathbb{N}}) > c$, then $(x_n)_{n \in \mathbb{N}}$ is $(1,c)$-large.
    \end{enumerate}
\end{lemma}

\begin{proof}
    If $(x_n)_{n \in \mathbb{N}}$ is $(1,c)$-large, then there is $M \in [\mathbb{N}]$ such that $\mathcal{S}_1^M \subseteq \mathcal{F}_c((x_n-x)_{n \in \mathbb{N}})$. We will show that $\widetilde{\operatorname{wu}}((x_n)_{n \in M}) \geq c$. Indeed, any subsequence of $(x_n)_{n \in M}$ is of the form $(x_n)_{n \in N}$ for some $N \in [M]$. It follows from the spreading property of $\mathcal{S}_1$ that $\mathcal{S}_1^N \subseteq \mathcal{S}_1^M$. Hence, $\mathcal{S}_1^N \subseteq \mathcal{F}_c((x_n-x)_{n \in \mathbb{N}})$ which is easily seen to be equivalent to saying $\mathcal{S}_1 \subseteq \mathcal{F}_c((x_n-x)_{n \in N})$. Therefore, as $\mathcal{S}_1$ contains sets of arbitrarily large cardinality, there is no $n \in \mathbb{N}$ such that
    \begin{align*}
        \forall x^* \in B_{X^*} : \; \# \{k \in N: \; |x^*(x_k-x)| \geq c\} \leq n
    \end{align*}
    and $\operatorname{wu}((x_n)_{n \in N}) \geq c$. It now follows from the definition that $\widetilde{\operatorname{wu}}((x_n)_{n \in M}) \geq c$.
    
    The other inequality follows from the proof of \cite[Lemma 1.13.]{Mercourakis1995} (note that we apply the lemma for $\delta = c$, $\Gamma = B_{X^*}$ and that the family $\mathcal{A}_\delta$ in this proof is nothing else than $\mathcal{F}_c$ in our notation). This lemma yields that if $\widetilde{\operatorname{wu}}((x_n)_{n \in \mathbb{N}}) > c$, then there is $M = (m_1,m_2,\dots) \in [\mathbb{N}]$ such that, if we set $M_k = (m_k,m_{k+1},\dots)$ for $k \in \mathbb{N}$, then $[M_k]^k \subseteq \mathcal{F}_c ((x_n-x)_{n \in \mathbb{N}})$. But
    \begin{align*}
        \bigcup_{k=1}^\infty [M_k]^k = \{(m_i)_{i \in F}: \; m_{|F|} \leq m_{\min F}\} = \{(m_i)_{i \in F}: \; |F| \leq \min F\} = \mathcal{S}_1^M. 
    \end{align*}
    Hence, $\mathcal{S}_1^M \subseteq \mathcal{F}_c ((x_n-x)_{n \in \mathbb{N}})$ and $(x_n)_{n \in \mathbb{N}}$ is $(1,c)$-large.
\end{proof}

\section{Quantitative characterization of weak $\xi$-Banach-Saks sets} \label{section:QuantWBS}

In this section we will prove a quantified version of the following theorem from \cite{ArgyrosMercourakusTsarpalias}, which is a natural generalization of the characterization of weakly null sequences with no Ces\`aro summable subsequences using $\ell_1$-spreading models (see \cite[Section 2]{LopezRuizTradacete2014}).

\begin{theorem}\cite[Theorem 2.4.1]{ArgyrosMercourakusTsarpalias} \label{TheoremAMT}
    Let $(x_n)_{n \in \mathbb{N}}$ be a weakly null sequence in a Banach space $X$ and $\xi < \omega_1$. Then exactly one of the following holds.
    \begin{enumerate}[(a)]
        \item For every $M \in [\mathbb{N}]$ there is $L \in [M]$ such that for every $P \in [L]$ the sequence $(x_n)_{n \in \mathbb{N}}$ is $(P,\xi)$-summable.
        \item There is $M = (m_n)_{n \in \mathbb{N}} \in [\mathbb{N}]$ such that the sequence $(x_{m_n})_{n \in \mathbb{N}}$ generates an $\ell_1^{\xi+1}$-spreading model.
    \end{enumerate}
\end{theorem}

More precisely, we will prove a formulation of the above-mentioned result for a bounded subset of a Banach space instead of a weakly null sequence and we will add the quantities $\operatorname{wbs}_\xi$ and $\operatorname{wus}_{\xi+1}$ (qualitative version of the quantity $\operatorname{wus}_{\xi+1}$ was also used in the proof of the Theorem \ref{TheoremAMT} from \cite{ArgyrosMercourakusTsarpalias}).

\begin{theorem}\label{TheoremQuantBS}
    Let $A$ be a bounded set in a Banach space $X$ and $\xi < \omega_1$. Then
    \begin{align}\tag{$\star$}\label{star}
        2\operatorname{sm}_{\xi+1} (A) \leq \operatorname{wbs}_\xi (A) \leq \operatorname{wbs}_\xi^s (A) \leq 2 \operatorname{wus}_{\xi+1} (A) \leq 4 \operatorname{sm}_{\xi+1} (A).
    \end{align}
\end{theorem}

As we have already mentioned, the inequality $\operatorname{wbs}_\xi(A) \leq \operatorname{wbs}_\xi^s(A)$ is trivial. Recall that $\mathcal{S}_\xi^M \subseteq \mathcal{S}_\xi [M]$ for all $M \in [\mathbb{N}]$ but in general we do not have equality. We have already noted that the support of $\xi_n^M$ is in the family $\mathcal{S}_\xi [M]$ for every $n \in \mathbb{N}$ and $M \in [\mathbb{N}]$. The following lemma from \cite{ArgyrosMercourakusTsarpalias} will allow us to find an infinite subset of $M$ for which the $\xi$-summability methods are very close to being supported on the sets from the smaller family $\mathcal{S}_\xi^M$. Note that in the following lemma the set $L$ does not depend on $\xi$.

\begin{lemma}\label{LemmaDeltaXi}
    \cite[Proposition 2.1.10.]{ArgyrosMercourakusTsarpalias}
    For every $M \in [\mathbb{N}]$ and $\epsilon > 0$ there is $L \in [M]$ such that for any $P \in [L]$, $\xi < \omega_1$ and $n \in \mathbb{N}$ there is $G \in \mathcal{S}_\xi^M$ such that
    \begin{align*}
        \langle \xi^P_n, G \rangle > 1 - \epsilon.
    \end{align*}
\end{lemma}

The following lemma provides sufficient conditions on an adequate family $\mathcal{F}$ so that the family $\mathcal{S}_{\xi+1}^N$, embeds into $\mathcal{F}$ for some $N \in [\mathbb{N}]$.

\begin{lemma}\label{LemmaVelka2}
    \cite[Theorem 2.2.6, Proposition 2.3.6]{ArgyrosMercourakusTsarpalias}
    Let $\mathcal{F}$ be an adequate family, $\xi < \omega_1$ and $\epsilon' > 0$. Suppose that there is $L = (l_n)_{n \in \mathbb{N}} \in [\mathbb{N}]$ satisfying
    \begin{itemize}
        \item For all $n \in \mathbb{N}$ and $N \in [L]$ with $l_n \leq \min N$ there is $F \in \mathcal{F}$ such that $\langle \xi_k^N, F \rangle > \epsilon'$ for $k = 1, \dots, n$.
    \end{itemize}
    Then there is $N \in [L]$ such that $\mathcal{S}_{\xi+1}^N \subseteq \mathcal{F}$.
\end{lemma}

Note that \cite[Proposition 2.3.6.]{ArgyrosMercourakusTsarpalias} has a slightly different formulation than Lemma \ref{LemmaVelka2}. More specifically, the condition on $L$ is formulated for all $N \in [L]$ with $n \leq \min N$ instead of $l_n \leq \min N$. However, the proof of \cite[Proposition 2.3.6.]{ArgyrosMercourakusTsarpalias} works for our formulation as well.

The following lemma is a quantified version of \cite[Lemma 2.4.8]{ArgyrosMercourakusTsarpalias}.

\begin{lemma}\label{LemmaRedukce}
    Let $(x_n)_{n \in \mathbb{N}}$ be a weakly null sequence in $B_X$, $\delta > 0$ and $\epsilon \in (0,1)$. Then for every $M \in [\mathbb{N}]$ there is $N \in [M]$ satisfying the following property:
    \begin{itemize}
        \item If $(a_n)_{n \in \mathbb{N}} \in S_{\ell_1}$, $\operatorname{supp} ((a_n)_{n \in \mathbb{N}}) \subseteq N$ and $F \in \mathcal{F}_{\delta}$, then
        $$ \norm{\sum_{n \in \mathbb{N}} a_n x_n} \geq (1-\epsilon) \: \delta \cdot \langle (a_n)_{n \in \mathbb{N}}, F \rangle - \epsilon \delta.$$
    \end{itemize}
\end{lemma}

\begin{proof}
    We use \cite[Lemma 2.4.7]{ArgyrosMercourakusTsarpalias} to find $N \in [M]$ such that for each $F \in \mathcal{F}_\delta [N]$ there is $x^* \in B_{X^*}$ such that
    \begin{enumerate}[(a)]
        \item $x^*(x_n) \geq (1-\epsilon)\delta$ for all $n \in F$,
        \item $\sum_{n \in N \setminus F} |x^*(x_n)| < \epsilon \delta$.
    \end{enumerate}
    Fix $(a_n)_{n \in \mathbb{N}} \in S_{\ell_1}$ with $\operatorname{supp} ((a_n)_{n \in \mathbb{N}}) \subseteq N$ and $F \in \mathcal{F}_{\delta}$. Take
    \begin{align*}
        F' = \{n \in F \cap N: \; a_n > 0\} \in \mathcal{F}_\delta [N].
    \end{align*}
    We find $x^* \in B_{X^*}$ such that the properties (a), (b) are satisfied for $F'$. Then
    \begin{align*}
        \norm{\sum_{n \in \mathbb{N}} a_n x_n} \geq \sum_{n \in N} a_n x^*(x_n) &\geq \sum_{n \in F'} a_n x^*(x_n) - \sum_{n \in N \setminus F'} |a_n x^*(x_n)| \\
        &\geq (1 - \epsilon) \delta \cdot \langle (a_n)_{n \in \mathbb{N}}, F' \rangle  - \epsilon \delta \\
        &\geq (1 - \epsilon) \delta \cdot \langle (a_n)_{n \in \mathbb{N}}, F \rangle  - \epsilon \delta.
    \end{align*}
    The inequality $\sum_{n \in N \setminus F'} |a_n x^*(x_n)| \leq \epsilon \delta$ follows from (b) and the fact that $|a_n| \leq 1$ for each $n \in \mathbb{N}$. The last inequality holds as $a_n \leq 0$ on $F \setminus F'$.
\end{proof}

We are all prepared to prove the first inequality of (\ref{star}). We will use the natural generalization of the idea of \cite[Lemma 4.5.]{Bendov__2015}.

\begin{prop}\label{PropSM->cca}
    Let $(x_n)_{n \in \mathbb{N}}$ be a sequence in a Banach space $X$ which weakly converges to some $x$. If $(x_n-x)_{n \in \mathbb{N}}$ generates an $\ell_1^{\xi+1}$-spreading model with a constant $c$, then
    \begin{align*}
        \widetilde{\operatorname{cca}}_\xi((x_{n^3})_{n \in \mathbb{N}}) \geq 2c.
    \end{align*}
\end{prop}

\begin{proof}
    We may without loss of generality assume that $(x_n)_{n \in \mathbb{N}} \subseteq B_X$ and that $x=0$. We will show that for every $P \in [\mathbb{N}]$ we have $\operatorname{cca} \left( \xi_n^P \cdot (x_{k^3})_{k \in \mathbb{N}} \right) \geq 2c$, and thus $\widetilde{\operatorname{cca}}_\xi((x_{n^3})_{n \in \mathbb{N}}) \geq 2c$.
    
    Take $l \in \mathbb{N}$ and set $n = l^2 + l$ and $m = l^3 +l$. For the sake of brevity we will write $z_j = \xi_j^P \cdot (x_{k^3})_{k \in \mathbb{N}}$. Then
    \begin{align*}
        \norm{\frac{1}{m} \sum_{j=1}^m z_j - \frac{1}{n} \sum_{j=1}^n z_j} = \norm{ \left( \frac{1}{m}- \frac{1}{n} \right) \sum_{j=1}^l z_j + \left( \frac{1}{m}- \frac{1}{n} \right) \sum_{j=l+1}^n z_j + \frac{1}{m} \sum_{j=n+1}^m z_j} \\
        \geq \norm{\left( \frac{1}{m}- \frac{1}{n} \right) \sum_{j=l+1}^n z_j + \frac{1}{m} \sum_{j=n+1}^m z_j} - \norm{ \left( \frac{1}{m}- \frac{1}{n} \right) \sum_{j=1}^l z_j}.
    \end{align*}
    It follows from the triangle inequality that
    \begin{align*}
        \norm{ \left( \frac{1}{m}- \frac{1}{n} \right) \sum_{j=1}^l z_j} \leq l \cdot \left( \frac{1}{n} - \frac{1}{m} \right) = \frac{l^3-l^2}{(l^2+l)(l^3+l)} \overset{l \rightarrow \infty}{\longrightarrow} 0.
    \end{align*}
    For $j \in \mathbb{N}$ set $\xi_j^P = (b^j_k)_{k \in \mathbb{N}}$ and $F_j = \operatorname{supp} \xi_j^P$. We define $F = \bigcup_{j=l+1}^m F_j$ and the finite sequence $(a_k)_{k \in F}$ by
    \begin{align*}
        a_k = \begin{cases}
                \left(\frac{1}{m} - \frac{1}{n} \right) b_k^j \; &\dots \;  \text{if } k \in F_j \text{ for } l+1 \leq j \leq n,\\
                \frac{1}{m} b_k^j \; &\dots \;  \text{if } k \in F_j \text{ for } n+1 \leq j \leq m.
            \end{cases}
    \end{align*}
    Then
    \begin{align*}
        \left( \frac{1}{m}- \frac{1}{n} \right) \sum_{j=l+1}^n z_j + \frac{1}{m} \sum_{j=n+1}^m z_j = \sum_{k \in F} a_k x_{k}.
    \end{align*}
    We have already observed that the sets $F_j$'s belong to the family $\mathcal{S}_\xi$. Hence, the sets $G_j = \{k^3: \; k \in F_j\}$ are also in the family $\mathcal{S}_\xi$ by its spreading property, and the set $G = \bigcup_{k=l+1}^m G_j$ is in $\mathcal{S}_{\xi+1}$ as $\min G_{l+1} \geq (l+1)^3 > l^3 = m-l$. Hence,
    \begin{align*}
        \norm{\sum_{k \in F} a_k x_{k^3}} = \norm{\sum_{k \in G} a_{\sqrt[3]{k}} x_k} \geq c \sum_{k \in G} |a_{\sqrt[3]{k}}| = c \sum_{k \in F} |a_k|
    \end{align*}
    as $(x_k)_{k \in \mathbb{N}}$ is an $\ell_1^{\xi+1}$-spreading model with constant $c$. Therefore we have
    \begin{align*}
        \norm{\left( \frac{1}{m}- \frac{1}{n} \right) \sum_{j=l+1}^n z_j + \frac{1}{m} \sum_{j=n+1}^m z_j} = \norm{\sum_{k \in F} a_k x_{k}} \\
        \geq c \sum_{k \in F} |a_k| = c \left( \sum_{j=l+1}^n \left( \frac{1}{n} - \frac{1}{m} \right) + \sum_{j=n+1}^m \frac{1}{m} \right) \\
        = c \left( \frac{(l^3-l^2)l^2}{(l^2+l)(l^3+l)} + \frac{l^3-l^2}{l^3+l} \right) \overset{l \rightarrow \infty}{\longrightarrow} 2 c.
    \end{align*}
    Hence,
    \begin{align*}
        \liminf_{l \rightarrow \infty} \norm{\frac{1}{m} \sum_{j=1}^m z_j - \frac{1}{n} \sum_{j=1}^n z_j}  \geq 2 c.
    \end{align*}
    It follows that $\operatorname{cca}(z_n) = \operatorname{cca} \left( \xi_n^P \cdot (x_{k^3})_{k \in \mathbb{N}} \right) \geq 2c$. Since $P \in [\mathbb{N}]$ was chosen arbitrarily, we get $\widetilde{\operatorname{cca}}_\xi((x_{n^3})_{n \in \mathbb{N}}) \geq 2c$.
\end{proof}

A version of the preceding proposition can be also shown using the approach of \cite{ArgyrosMercourakusTsarpalias}. However, the best result we were able to get using this approach was $\widetilde{\operatorname{cca}}_\xi^s((x_n)_{n \in \mathbb{N}}) \geq c$. The approach of \cite{Bendov__2015} is more elementary and gives a better constant. Note that the first inequality of (\ref{star}) from Theorem \ref{TheoremQuantBS} is an immediate consequence of Proposition \ref{PropSM->cca}. We proceed with proving the third inequality, using the approach of \cite{ArgyrosMercourakusTsarpalias}.

\begin{prop}\label{Propcca->large}
    Let $(x_n)_{n \in \mathbb{N}}$ be a sequence in a Banach space $X$ which weakly converges to some $x \in X$ and let $c>0$ and $\xi<\omega_1$. Suppose that $\widetilde{\operatorname{cca}}_\xi^s ((x_n)_{n \in \mathbb{N}}) > c$. Then $(x_n)_{n \in \mathbb{N}}$ is $(\xi+1, \frac{c}{2})$-large.
\end{prop}

\begin{proof}
    We can assume that $x = 0$ and $(x_n)_{n \in \mathbb{N}} \subseteq B_X$. Take $c' > c$ such that $\widetilde{\operatorname{cca}}_\xi^s ((x_n)_{n \in \mathbb{N}}) > c'$ and fix $\epsilon > 0$ small enough so that $(1-2\epsilon) c' \geq c$.

    As $\widetilde{\operatorname{cca}}_\xi^s ((x_n)_{n \in \mathbb{N}}) > c'$, we can find $M \in [\mathbb{N}]$ such that for all $N \in [M]$ we have $\operatorname{cca}(\xi_n^N \cdot (x_k)_{k \in \mathbb{N}}) > c'$. It follows from the triangle inequality that for all $N \in [M]$ we have
    \begin{align*}
        \limsup_{n \rightarrow \infty} \norm{ \frac{1}{n} \sum_{i=1}^n \xi_i^N \cdot (x_k)_{k \in \mathbb{N}}} > \frac{c'}{2}.
    \end{align*}
    
    Now we will recursively construct a sequence $(L_n)_{n \in \mathbb{N}}$ of infinite subsets of $M$ such that
    \begin{enumerate}[(a)]
        \item $L_1 \in [M]$, $L_n \in [L_{n-1}]$ for $n \geq 2$.
        \item For every $n \in \mathbb{N}$ and $N \in [L_n]$ there is $x^* \in B_{X^*}$ with
        \begin{align*}
            x^* \left( \xi_j^N \cdot (x_k)_{k \in \mathbb{N}} \right) > (1-\epsilon) \frac{c'}{2}, \hspace{2cm} j=1,\dots,n.
        \end{align*}
    \end{enumerate}
    We shall proceed by induction over $n \in \mathbb{N}$. For convenience we set $L_0 = M$. Let us assume that $L_{n-1}$ was already defined. We partition $[L_{n-1}]$ into two subsets
    \begin{align*}
        &A_1 = \left\{ P \in [L_{n-1}]: \; \exists x^* \in B_{X^*} \text{ such that} \; x^* \left( \xi^P_j \cdot (x_k)_{k \in \mathbb{N}} \right) > (1-\epsilon) \frac{c'}{2}, \; j \leq n \right\} \\
        &A_2 = [L_{n-1}] \setminus A_1.
    \end{align*}
    The set $A_1$ is open. Indeed, for a fixed set $P \in A_1$ the sets $P' \in [L_{n-1}]$ for which $\chi_P (j) = \chi_{P'}(j)$ for $j = 1, \dots, \max \operatorname{supp} \xi_{n}^P$ form a neighbourhood of $P$ in $[L_{n-1}]$ which is contained in $A_1$ as $\xi_j^P = \xi_j^{P'}$, $j=1,\dots,n$, for such sets by P.3 in \cite[page 171]{ArgyrosMercourakusTsarpalias}. Hence, $A_1$ is a Borel set, and thus a completely Ramsey set. By the infinite Ramsey theorem \cite[Theorem 10.1.3.]{nigel2006topics} there is $L_n \in [L_{n-1}]$ such that either $[L_n] \subseteq A_1$ or $[L_n] \subseteq A_2$. We will show that the second case is not possible.
    
    We recall that $\limsup_s \norm{ \frac{1}{s} \sum_{i=1}^s \xi_i^{L_n} \cdot (x_k)_{k \in \mathbb{N}}} > \frac{c'}{2}$. Therefore, we can find large enough $s \in \mathbb{N}$ and $x^* \in B_{X^*}$ such that the following two conditions are satisfied.
    \begin{align*}
        x^* \left( \frac{1}{s} \sum_{i=1}^s \xi_i^{L_n} \cdot (x_k)_{k \in \mathbb{N}} \right) >  \frac{c'}{2}, \hspace{2cm} s \epsilon \frac{c'}{2} \geq n.
    \end{align*}
    Set
    \begin{align*}
       I_1 &= \left\{1 \leq i \leq s: \; x^*(\xi_i^{L_n} \cdot (x_k)_{k \in \mathbb{N}}) > (1-\epsilon) \frac{c'}{2} \right\} \\
       I_2 &= \{1,\dots,s\} \setminus I_1.
    \end{align*}
    Then
    \begin{align*}
        \frac{c'}{2} &< \frac{1}{s} \sum_{i=1}^s x^* \left(\xi_i^{L_n} \cdot (x_k)_{k \in \mathbb{N}} \right) \\ &= \frac{1}{s} \left( \sum_{i \in I_1} x^* \left(\xi_i^{L_n} \cdot (x_k)_{k \in \mathbb{N}} \right) + \sum_{i \in I_2} x^* \left(\xi_i^{L_n} \cdot (x_k)_{k \in \mathbb{N}} \right) \right) \\
        &\leq \frac{1}{s} \left(|I_1| + s \: (1-\epsilon) \frac{c'}{2} \right).
    \end{align*}
    But that implies
    \begin{align*}
        |I_1| > s \epsilon \frac{c'}{2} \geq n.
    \end{align*}
    Hence, we can find $i_1 < i_2 < \cdots < i_n \leq s$ satisfying $x^*(\xi_{i_j}^{L_n} \cdot (x_k)_{k \in \mathbb{N}}) > (1-\epsilon) \frac{c'}{2}$. But now we can pick $P \in [L_{n}]$, such that $\xi^{L_{n}}_{i_j} = \xi^{P}_j$, see P.4 in \cite[page 171]{ArgyrosMercourakusTsarpalias}, and for this $P$ we have $P \in A_1$. Hence, $[L_n] \not\subseteq A_2$, and therefore $[L_n] \subseteq A_1$.
    
    Now we take a diagonal subsequence $L = (l_k)_{k \in \mathbb{N}}$ of the sequences $(L_k)_{k \in \mathbb{N}}$. For all $n \in \mathbb{N}$ and $P \in [L]$ with $l_n \leq \min P$ there is some $x^* \in B_{X^*}$ such that
    \begin{align*}
        x^*(\xi_i^{P} \cdot (x_k)_{k \in \mathbb{N}}) > (1-\epsilon) \frac{c'}{2} \hspace{2cm} \text{for } i=1,\dots,n.
    \end{align*}
    Indeed, as $l_n \leq \min P$, we have $P \in [L_n]$ and are done by property (b).
    
    Let us denote for brevity $c'' = (1- 2\epsilon) \frac{c'}{2}$. We take a further subset $P = (p_n)_{n \in \mathbb{N}} \in [L]$ such that the conclusion of Lemma \ref{LemmaRedukce} is satisfied on $P$ for $\delta = c''$ and $\epsilon$.
    
    We will show that $\mathcal{F}_{c''}$ and $P$ satisfy the assumptions of Lemma \ref{LemmaVelka2} for $\epsilon' = \epsilon\frac{c'}{2}$. That is, we want to show that for every $n \in \mathbb{N}$ and $P' \in [P]$ with $p_n \leq \min P'$ there is $F \in \mathcal{F}_{c''}$ with $\langle \xi_i^{P'}, F \rangle > \epsilon'$ for $i=1,\dots,n$. Take such $n \in \mathbb{N}$ and $P' \in [P]$. As $P' \in [L]$ and $l_n \leq p_n \leq \min P'$, we can find some $x^* \in B_{X^*}$ such that
    \begin{align*}
        x^* \left( \xi_i^{P'} \cdot (x_k)_{k \in \mathbb{N}} \right) > (1-\epsilon) \frac{c'}{2} = c'' + \epsilon', \hspace{2cm} i=1,\dots,n.
    \end{align*}
    But then for $F = \{n \in \mathbb{N}: \; x^*(x_n) > c''\} \in \mathcal{F}_{c''}$ we have that $\langle \xi_i^{P'}, F \rangle > \epsilon'$ for $i = 1,\dots,n$ as otherwise, if we set $\xi_i^{P'} = (b_k)_{k \in \mathbb{N}}$, we would get the following contradiction
    \begin{align*}
        c'' + \epsilon' < \sum_{k \in \mathbb{N}} b_k x^*(x_k) &= \sum_{k \in F} b_k x^*(x_k) + \sum_{k \in \mathbb{N} \setminus F} b_k x^*(x_k) \\
        &\leq \sum_{k \in F} b_k + \sum_{k \in \mathbb{N} \setminus F} b_k c'' \leq \epsilon' + c''.
    \end{align*}
    Hence, the assumptions of Lemma \ref{LemmaVelka2} are satisfied and we can find $Q = (q_i)_{i \in \mathbb{N}} \in [P]$ such that $\mathcal{S}_{\xi+1}^Q \subseteq \mathcal{F}_{c''}$. But this means that $(x_n)_{n \in \mathbb{N}}$ is $(\xi+1,c'')$-large. Recall that
    \begin{align*}
        c'' = (1-2\epsilon) \frac{c'}{2} \geq \frac{c}{2},
    \end{align*}
    and thus $\mathcal{F}_{c''} \subseteq \mathcal{F}_{\frac{c}{2}}$ and $(x_n)_{n \in \mathbb{N}}$ is also $(\xi+1,\frac{c}{2})$-large.
\end{proof}

The third inequality of (\ref{star}) from Theorem \ref{TheoremQuantBS} follows from Proposition \ref{Propcca->large}. We finish the proof of Theorem \ref{TheoremQuantBS} by proving the last inequality, for which we also use the approach of \cite{ArgyrosMercourakusTsarpalias}.

\begin{prop} \label{Proplarge->sm}
    Let $(x_n)_{n \in \mathbb{N}}$ be a sequence in a Banach space $X$ which weakly converges to some $x \in X$. Let $c>0$ and $\xi < \omega_1$ be such that $(x_n)_{n \in \mathbb{N}}$ is $(\xi,c)$-large. Then for any $d< \frac{c}{2}$ there is $N \in [\mathbb{N}]$ such that $(x_n-x)_{n \in N}$ generates an $\ell_1^{\xi}$-spreading model with constant $d$.
\end{prop}

\begin{proof}
    Without loss of generality we can assume that $x = 0$ and $(x_n)_{n \in \mathbb{N}} \subseteq B_{X}$. As $(x_n)_{n \in \mathbb{N}}$ is $(\xi,c)$-large, there is $M \in [\mathbb{N}]$ such that $\mathcal{S}_{\xi}^M \subseteq \mathcal{F}_c$. We can take $\epsilon > 0$ small enough such that $(1-\epsilon) \frac{c}{2} - \epsilon c \geq d$. We will use Lemma \ref{LemmaRedukce} to find $N = (n_k)_{k \in \mathbb{N}} \in [M]$ such that the conclusion of Lemma \ref{LemmaRedukce} is satisfied on $N$ for $\epsilon$ and $\delta = c$.
    
    Now we will show that $(x_n)_{n \in N}$ generates an $\ell_1^{\xi}$-speading model with constant $d$. Fix $F \in \mathcal{S}_{\xi}$ and a sequence of scalars $(b_k)_{k \in F}$. We can assume without loss of generality that $\sum_{k \in F} |b_k| = 1$. Then it is enough to show that
    \begin{align*}
        \norm{\sum_{k \in F} b_k x_{n_k}} \geq d.
    \end{align*}
    We define the sequence of scalars $(a_k)_{k \in \mathbb{N}}$ by the rule $a_j = b_k$, if $j = n_k$ for some $k \in F$, and $a_j = 0$ otherwise. Then $(a_k)_{k \in \mathbb{N}} \in S_{\ell_1}$ and $\operatorname{supp} ((a_k)_{k \in \mathbb{N}}) \subseteq N$. Hence, we get that the following inequality holds for any $G \in \mathcal{F}_{c}$.
    \begin{align*}
        \norm{\sum_{k \in \mathbb{N}} a_k x_k} \geq (1-\epsilon) c \cdot \langle (a_n)_{n \in \mathbb{N}}, G \rangle - \epsilon c.
    \end{align*}
    We can also assume, if we define $F^+ = \{k \in F: \; b_k > 0\}$ and $F^- = \{k \in F: \; b_k < 0\}$, that $\sum_{k \in F^+} |b_k| \geq \frac{1}{2}$. If not, we can consider $(-b_k)_{k \in F}$ instead of $(b_k)_{k \in F}$. Note that $G = \{n_k: \; k \in F^+\} \in \mathcal{S}^N_{\xi} \subseteq \mathcal{F}_{c}$ as $F^+ \in \mathcal{S}_{\xi}$. Then we have
    \begin{align*}
        \langle (a_n)_{n \in \mathbb{N}}, G \rangle = \sum_{k \in F^+} a_{n_k} = \sum_{k \in F^+} b_k \geq \frac{1}{2}.
    \end{align*}
    Therefore
    \begin{align*}
        \norm{\sum_{k \in F} b_k x_{n_k}} &= \norm{\sum_{k \in \mathbb{N}} a_k x_k} \geq (1-\epsilon) \frac{c}{2} - \epsilon c \geq d.
    \end{align*}
\end{proof}

\begin{remark}
    The proof of Theorem \ref{TheoremQuantBS} combines the approach of \cite{Bendov__2015}, which is generalised for arbitrary $\xi < \omega_1$ and used to prove Proposition \ref{PropSM->cca}, and the approach of \cite{ArgyrosMercourakusTsarpalias}, which is used to prove Propositions \ref{Propcca->large} and \ref{Proplarge->sm}. More precisely, the proofs of Propositions \ref{Propcca->large} and \ref{Proplarge->sm} mimic the proof of \cite[Theorem 2.4.1]{ArgyrosMercourakusTsarpalias} with quantitative interpretation of \cite[Lemmata 2.4.3, 2.4.8]{ArgyrosMercourakusTsarpalias}. We also needed Lemmata \ref{LemmaDeltaXi} and \ref{LemmaVelka2} (that is \cite[Propositions 2.1.10, 2.3.6 and Theorem 2.2.6]{ArgyrosMercourakusTsarpalias}, but these results offer no quantitative improvement, and so are presented here without proof.
\end{remark}

Now we will prove two corollaries to Theorem \ref{TheoremQuantBS}. The first one is that the quantity $\operatorname{wbs}_\xi$ indeed characterizes weak $\xi$-Banach-Saks sets.

\begin{prop}\label{PropWBS}
    Let $A$ be a bounded set in a Banach space $X$ and $\xi<\omega_1$. Then $A$ is a weak $\xi$-Banach-Saks set, if and only if $\operatorname{wbs}_\xi (A) = 0$.
\end{prop}

\begin{proof}
    It is straightforward that if $\operatorname{wbs}_\xi(A) > 0$ then $A$ is not a weak $\xi$-Banach-Saks set. On the other hand, suppose that $\operatorname{wbs}_\xi (A) = 0$ and fix a sequence $(x_n)_{n \in \mathbb{N}}$ in $A$ which is weakly convergent to some $x \in X$. It follows from Theorem \ref{TheoremQuantBS} that $\operatorname{sm}_{\xi+1}(A) = 0$, and therefore $(x_n - x)_{n \in \mathbb{N}}$ contains no subsequence that generates an $\ell_1^{\xi+1}$-spreading model. Hence, by Theorem \ref{TheoremAMT}, we get that for every $M \in [\mathbb{N}]$ there is $L \in [M]$ such that for all $P \in [L]$ the sequence $(x_n-x)_{n \in \mathbb{N}}$, and thus also the sequence $(x_n)_{n \in \mathbb{N}}$, is $(P,\xi)$-summable. Therefore, we can take $M = \mathbb{N}$ and $P = L$, and get that $A$ is a weak $\xi$-Banach-Saks set.
\end{proof}

The second corollary shows that weak $\xi$-Banach-Saks sets enjoy a formally stronger property analogous to the fact that any weakly convergent sequence in a weak Banach-Saks set admits a subequence with every further subsequence being Ces\`aro summable (indeed, in this case it is enough to consider a uniformly weakly convergent subsequence). Note that the following proposition is, in essence, a qualitative version of the inequalities $\operatorname{wbs}_\xi(A) \leq \operatorname{wbs}_\xi^s(A) \leq 2\operatorname{wbs}_\xi(A)$ from Theorem \ref{TheoremQuantBS}.

\begin{prop}\label{PropWBSEkvivalentniDef}
    Let $A$ be a bounded subset of a Banach space $X$. Then the following are equivalent:
    \begin{enumerate}[(a)]
        \item For every weakly convergent sequence $(x_n)_{n \in \mathbb{N}}$ in $A$ and every $M \in [\mathbb{N}]$ there is $N \in M$ such that for all $P \in [N]$ the sequence $(x_n)_{n \in \mathbb{N}}$ is $(P,\xi)$-summable;
        \item $A$ is a weak $\xi$-Banach-Saks set.
    \end{enumerate}
\end{prop}

\begin{proof}
    The fact that $(a)$ implies $(b)$ follows immediately from the definitions. We will show the other implication. Suppose that $(a)$ does not hold. Then by Theorem \ref{TheoremAMT} there is a sequence $(x_n)_{n \in \mathbb{N}}$ in $A$ which converges weakly to some $x \in X$ such that $(x_n-x)_{n \in \mathbb{N}}$ generates an $\ell_1^{\xi+1}$-spreading model with constant $c$ for some $c>0$. But then $\operatorname{sm}_{\xi+1}(A) \geq c$ and thus by Theorem \ref{TheoremQuantBS} $\operatorname{wbs}_\xi(A) \geq 2c$ and $A$ cannot be a weak $\xi$-Banach-Saks set by Proposition \ref{PropWBS}.
\end{proof}

\section{$\xi$-Banach-Saks sets and compactness} \label{section:QuantBS}

Following \cite{Bendov__2015}, in this section we will show the quantitative interpretation of the following implications for a bounded subset $A$ of a Banach space $X$ and $\xi < \omega_1$:
\begin{align*}
    \begin{array}{c}
        A \text{ is relatively norm compact} \\
        \Downarrow \\
        A \text{ is a } \xi\text{-Banach-Saks set } \\
        \Downarrow \\
        A \text{ is relatively weakly compact and a weak } \xi\text{-Banach-Saks set}.
    \end{array}
\end{align*}

Note that the second implication can be reversed but the converse implication cannot be quantified for $\xi = 0$, as illustrated by \cite[Example 3.3.]{Bendov__2015}.

We will first define the quantities measuring weak and norm non-compactness.

\begin{definition}
    Let $(x_n)_{n \in \mathbb{N}}$ be a bounded sequence in a Banach space $X$. We define the quantity
    \begin{align*}
        \widetilde{\operatorname{ca}}(x_n) = \inf \{\operatorname{ca}(y_n): \; (y_n)_{n \in \mathbb{N}} \text{ is a subsequence of } (x_n)_{n \in \mathbb{N}}\}.
    \end{align*}

    Let $A$ be a bounded subset of a Banach space $X$. We define
    \begin{align*}
        \beta(A) &= \sup \{ \widetilde{\operatorname{ca}}(x_n): \; (x_n)_{n \in \mathbb{N}} \text{ is a sequence in }A\} \\
        \operatorname{wck}_X (A) &= \operatorname{sup} \{ \operatorname{d}(\operatorname{clust}_{X^{**}}(x_n),X): \; (x_n)_{n \in \mathbb{N}} \text{ is a sequence in }A\},
    \end{align*}
    where $\operatorname{d} (B,C) = \inf \{\norm{b-c}: \; b \in B, c \in C\}$ is the standard distance of sets and $\operatorname{clust}_{X^{**}}(x_n)$ is the set of all weak$^*$ cluster points of the sequence $(x_n)_{n \in \mathbb{N}}$ in the space $X^{**}$.
\end{definition}

Note that the quantity $\beta$ indeed measures non-compactness and the quantity $\operatorname{wck}_X$ indeed measures weak non-compactness. That is, $\beta(A) = 0$ if and only if $A$ is relatively compact and $\operatorname{wck}_X (A) = 0$ if and only if $A$ is relatively weakly compact. For more information about these quantities and their relation to other quantities see \cite{Bendov__2015}. Now we are all prepared to prove the following theorem.

\begin{theorem}\label{TheoremBSaWBS}
    Let $A$ be a bounded subset of a Banach space $X$ and $\xi < \omega_1$. Then
    \begin{align*}
        \max \{\operatorname{wck}_X(A), \operatorname{wbs}_\xi(A)\} \leq \operatorname{bs}_\xi(A) \leq \operatorname{bs}_\xi^s(A) \leq \beta (A).
    \end{align*}
\end{theorem}

To prove the inequality $\operatorname{wck}_X (A) \leq \operatorname{bs}_\xi (A)$ we will need to define an auxiliary quantity $\gamma_0$. For a bounded subset $A$ of a Banach space $X$ we define
\begin{align*}
    \gamma_0(A) = \sup \{|\lim_{m \rightarrow \infty}& \lim_{n \rightarrow \infty} x_m^*(x_n)|: \\
    &(x_m)^*_{m \in \mathbb{N}} \text{ is a weak}^*\text{ null sequence in } B_{X^*},\\
    &(x_n)_{n \in \mathbb{N}} \text{ is a sequence in }A \\
    &\text{and all the involved limits exist} \}.
\end{align*}

The quantity $\gamma_0$ was introduced in \cite{CascalesKalendaSpurny} as a measure of weak compactness in spaces whose duals have weak$^*$ angelic unit balls. Later, it was used \cite{Bendov__2015} to prove a version of Theorem \ref{TheoremBSaWBS} for $\xi = 1$.

\begin{lemma}\label{lemmaGamma0}
    Let $A$ be a bounded subset of a Banach space $X$ and $\xi < \omega_1$. Then
    \begin{align*}
        \gamma_0(A) \leq \operatorname{bs}_\xi(A).
    \end{align*}
\end{lemma}

\begin{proof}
    Suppose that $\gamma_0(A) > c$ for some $c>0$. Then there is a sequence $(x_k)_{x \in \mathbb{N}}$ in $A$ and a weak$^*$ null sequence $(x_n^*)_{n \in \mathbb{N}}$ such that
    \begin{align*}
        \lim_{j \rightarrow \infty} \lim_{k \rightarrow \infty} x_j^*(x_k) > c.
    \end{align*}
    We can assume without loss of generality that $\lim_{k \rightarrow \infty} x_j^*(x_k) > c$ for all $j \in \mathbb{N}$. Fix $P \in [\mathbb{N}]$ and define, for $k \in \mathbb{N}$,
    \begin{align*}
        y_k = \frac{1}{k} \sum_{j=1}^k \xi_j^P \cdot (x_n)_{n \in \mathbb{N}}.
    \end{align*}
    We want to show that $\operatorname{ca}(y_n) \geq c$. Note that for each $j \in \mathbb{N}$ we have
    \begin{align*}
        \lim_{k \rightarrow \infty} x_j^*(y_k) = \lim_{k \rightarrow \infty} x_j^*(x_k) > c.
    \end{align*}
    Now fix $\epsilon > 0$ and $k \in \mathbb{N}$. Using weak$^*$ nullness of the sequence $(x_j^*)_{j \in \mathbb{N}}$, we can find $j \in \mathbb{N}$ such that $x_j^*(y_k) < \epsilon$. Then we can find $l > k$ such that $x_j^*(y_l) > c$. But then
    \begin{align*}
        \norm{y_l - y_k} \geq x_j^*(y_l-y_k) > c-\epsilon.
    \end{align*}
    As $\epsilon$ and $k$ were chosen arbitrarily, we get that $\operatorname{ca}(y_n) \geq c$. As $P \in [\mathbb{N}]$ was also chosen arbitrarily, we get $\widetilde{\operatorname{cca}}_\xi((x_n)_{n \in \mathbb{N}}) \geq c$, and this implies that $\operatorname{bs}_\xi(A) \geq c$.
\end{proof}

\begin{proof}[Proof of Theorem \ref{TheoremBSaWBS}]
    We first note that the inequality $\operatorname{bs}_\xi(A) \leq \operatorname{bs}_\xi^s(A)$ is trivial. We proceed with the first inequality. That $\operatorname{bs}_\xi(A) \geq \operatorname{wbs}_\xi(A)$ is clear. If $X$ is separable, then the closed unit ball of $X^*$ is metrizable and $\gamma_0(A) = \operatorname{wck}_X(A)$ by \cite[Theorem 6.1.]{CascalesKalendaSpurny}. Hence, for separable $X$ we get the inequality $\operatorname{bs}_\xi(A) \geq \gamma_0(A) = \operatorname{wck}_X(A)$.
    
    If $X$ is arbitrary and $\operatorname{wck}_X(A) > c$ for some $c > 0$, we can find a sequence $(x_k)_{k \in \mathbb{N}}$ in $A$ with $\operatorname{d}(\operatorname{clust}_{X^{**}}(x_k),X) > c$. If we set $Y = \overline{\operatorname{span}} \{x_k: \; k \in \mathbb{N}\}$, then $Y$ is a separable subspace of $X$ and $\operatorname{d}(\operatorname{clust}_{Y^{**}}(x_k),Y) \geq \operatorname{d}(\operatorname{clust}_{X^{**}}(x_k),X)$ (see the proof of \cite[Theorem 3.1.]{Bendov__2015}). Therefore
    \begin{align*}
        \operatorname{d}( \operatorname{clust}_{Y^{**}} (x_k),Y) \geq \operatorname{d}(\operatorname{clust}_{X^{**}}(x_k),X) > c.
    \end{align*}
    It follows that $\operatorname{wck}_Y(A \cap Y) > c$, and therefore
    \begin{align*}
        \operatorname{bs}_\xi(A) \geq \operatorname{bs}_\xi(A \cap Y) \geq \operatorname{wck}_Y(A \cap Y) > c
    \end{align*}
    by the already proved separable case.
    
    The last inequality we need to prove is $\operatorname{bs}^s_\xi(A) \leq \beta(A)$. For this we use to following lemma.
    
    \begin{lemma} \label{LemmaCa->Cca}
        Let $(x_n)_{n \in \mathbb{N}}$ be a bounded sequence in a Banach space $X$ and $\xi < \omega_1$. Let there be $c > 0$ and $N \in [\mathbb{N}]$ such that $\operatorname{ca}((x_n)_{n \in N}) < c$. Then for any $P \in [N]$ we have $\operatorname{cca} (\xi_n^P \cdot (x_k)_{k \in \mathbb{N}}) \leq c$.
    \end{lemma}
    
    \begin{proof}
        As $\operatorname{ca}((x_n)_{n \in N}) < c$, we can find $n_0 \in \mathbb{N}$ such that
        \begin{align*}
            \norm{x_n - x_m} \leq c, \hspace{2cm} \text{for } n,m \in N \text{ and } n,m > n_0.
        \end{align*}
        We define $y_n = \xi_n^P \cdot (x_k)_{k \in \mathbb{N}}$ for $n \in \mathbb{N}$. Note that
        \begin{align*}
            \norm{y_n-y_m} \leq c, \hspace{2cm} \text{for } n,m > n_0.
        \end{align*}
        To prove it we notice that $\norm{x_n - y_m} \leq c$ for each $n,m > n_0$, $n \in N$ as such $y_m$ is a convex combination of elements $x_j$'s for which $\norm{x_n-x_j} \leq c$. Hence, for $n>n_0$ we have that $y_n$ is a convex combination of elements $x_j$'s for which $\norm{x_j-y_m} \leq c$ for each $m>n_0$, and thus also $\norm{y_n-y_m} \leq c$ for each $m > n_0$. Hence, $\operatorname{ca}((y_n)_{n \in \mathbb{N}}) \leq c$ and by \cite[Lemma 3.4.]{Bendov__2015} $\operatorname{cca} (\xi_n^P \cdot (x_k)_{k \in \mathbb{N}}) = \operatorname{cca}(y_n) \leq c$.
    \end{proof}
    The only inequality left is $\beta(A) \geq \operatorname{bs}_\xi^s(A)$. Let $\beta(A) < c$ for some $c>0$ and take an arbitrary sequence $(x_n)_{n \in \mathbb{N}}$ in $A$. What we want to show is $\widetilde{\operatorname{cca}}_\xi^s((x_n)_{n \in \mathbb{N}}) \leq c$. Let $M \in [\mathbb{N}]$ be arbitrary, then we can find $N \in [M]$ such that $\operatorname{ca}((x_n)_{n \in M}) < c$. It then follows from Lemma \ref{LemmaCa->Cca} that for any $P \in [N]$ we have $\operatorname{cca}(\xi_n^P \cdot (x_k)_{k \in \mathbb{N}}) \leq c$. In particular, $\operatorname{cca}(\xi_n^N \cdot (x_k)_{k \in \mathbb{N}}) \leq c$. As $M$ was arbitrary, $\widetilde{\operatorname{cca}}_\xi^s((x_n)_{n \in \mathbb{N}}) \leq c$. As $(x_n)_{n \in \mathbb{N}}$ was also chosen arbitrarily, $\operatorname{bs}_\xi^s(A) \leq c$ and we are done. Thus, Theorem \ref{TheoremBSaWBS} is proved.
\end{proof}

In the following propositions we show the converse to the second implication mentioned at the beginning of this section, that is that a relatively weakly compact weak $\xi$-Banach-Saks set is a $\xi$-Banach-Saks set. As mentioned, this implication cannot be fully quantified.

\begin{prop}\label{PropCharBS}
    Let $A$ be relatively weakly compact subset of a Banach space $X$ and $\xi < \omega_1$. Then $\operatorname{wbs}_\xi^s(A) = \operatorname{bs}_\xi^s (A)$.
\end{prop}

\begin{proof}
    For any bounded set $A$ we have $\operatorname{wbs}_\xi^s (A) \leq \operatorname{bs}_\xi^s (A)$. For the converse, let $(x_n)_{n \in \mathbb{N}}$ be a sequence in $A$ and $M \in [\mathbb{N}]$. As $A$ is relatively weakly compact, we can use the Eberlein-Šmulyan theorem to find $N = (n_k)_{k \in \mathbb{N}} \in [M]$ such that $(x_k)_{k \in N}$ is weakly convergent to some $x \in X$. Denote by $N^c = \mathbb{N} \setminus N$. We define $y_k = x_{n_k}$, for $k \in N^c$, and $y_k = x_k$, for $k \in N$. Then $(y_k)_{k \in \mathbb{N}}$ is a sequence in $A$ weakly converging to $x$. Hence, for any $\epsilon > 0$ there is $L_\epsilon \in [N]$ such that $\operatorname{cca} (\xi_n^{L_\epsilon} \cdot (y_k)_{k \in \mathbb{N}}) \leq \operatorname{wbs}_\xi^s (A) + \epsilon$. But $\xi_n^{L_\epsilon} \cdot (y_k)_{k \in \mathbb{N}} = \xi_n^{L_\epsilon} \cdot (x_k)_{k \in \mathbb{N}}$, as $y_k = x_k$ for $k \in L_\epsilon \subseteq N$. Thus, as $M \in [\mathbb{N}]$ and $\epsilon > 0$ were arbitrary, we have shown that $\widetilde{\operatorname{cca}}_\xi^s ((x_n)_{n \in \mathbb{N}}) \leq \operatorname{wbs}_\xi^s (A)$. As $(x_n)_{n \in \mathbb{N}}$ was arbitrary, we get $\operatorname{bs}_\xi^s (A) \leq \operatorname{wbs}_\xi^s (A)$.
\end{proof}

We can use the same trick (that is replacing a bounded sequence $(x_n)_{n \in \mathbb{N}}$ with a weakly convergent sequence $(y_n)_{n \in \mathbb{N}}$ as in the proof of Proposition \ref{PropCharBS}) to prove the promised converse to the second implication mentioned at the beginning of this section as well as an analogue of Proposition \ref{PropWBSEkvivalentniDef} for the $\xi$-Banach-Saks property.

\begin{prop}\label{PropEqBsWbs}
    Let $\xi < \omega_1$ and $A$ be a bounded set in a Banach space $X$. Then the following are equivalent:
    \begin{enumerate}[(i)]
        \item $A$ is a $\xi$-Banach-Saks set;
        \item For every sequence $(x_n)_{n \in \mathbb{N}}$ in $A$ and every $M \in [\mathbb{N}]$ there is $L \in [M]$ such that for all $P \in [L]$ the sequence $(x_n)_{n \in \mathbb{N}}$ is $(P,\xi)$-summable;
        \item $A$ is a relatively weakly compact weak $\xi$-Banach-Saks set.
    \end{enumerate}
\end{prop}

\begin{proof}
    If $A$ ia a $\xi$-Banach-Saks set, then it is trivially a weak $\xi$-Banach-Saks set. Further $\operatorname{bs}_\xi(A) = 0$, and thus $A$ is relatively weakly compact by Theorem \ref{TheoremBSaWBS}. Hence, (i) implies (iii). Clearly, (ii) implies (i).
    
    What is left is the implication (iii) implies (ii). Let us suppose that $A$ is a relatively weakly compact weak $\xi$-Banach-Saks set. Let $(x_n)_{n \in \mathbb{N}}$ be a sequence in $A$ and $M \in[ \mathbb{N}]$. It follows from the Eberlein-Šmulyan theorem that there is $N \in [M]$ such that $(x_n)_{n \in N}$ is weakly convergent. We define the sequence $(y_n)_{n \in \mathbb{N}}$ in exactly the same way as in the proof of Proposition \ref{PropCharBS}. Then $(y_n)_{n \in \mathbb{N}}$ is a weakly convergent sequence in the weak $\xi$-Banach-Saks set $A$, and thus by Proposition \ref{PropWBSEkvivalentniDef} there is $L \in [N]$ such that for all $P \in [L]$ the sequence $(y_n)_{n \in \mathbb{N}}$ is $(P,\xi)$-summable. But then again we have $x_k = y_k$ for $k \in L$, and therefore the sequence $(x_n)_{n \in \mathbb{N}}$ is also $(P,\xi)$-summable. Hence, we have found for any sequence $(x_n)_{n \in \mathbb{N}}$ in $A$ and $M \in [\mathbb{N}]$ a further subset $L \in [M]$ such that for all $P \in [L]$ the sequence $(x_n)_{n \in \mathbb{N}}$ is $(P,\xi)$-summable and (ii) holds.
\end{proof}

In the following proposition we prove that both of the quantities $\operatorname{bs}_\xi$ and $\operatorname{bs}_\xi^s$ quantify the $\xi$-Banach-Saks property.

\begin{prop}\label{PropBS}
    Let $A$ be a bounded set in a Banach space $X$ and $\xi < \omega_1$. Then $A$ is a $\xi$-Banach-Saks set, if and only if $\operatorname{bs}_\xi (A) = 0$, if and only if $\operatorname{bs}_\xi^s(A) = 0$.
\end{prop}

\begin{proof}
    If $\operatorname{bs}_\xi^s(A) = 0$, then trivially $\operatorname{bs}_\xi(A) = 0$. If $\operatorname{bs}_\xi (A) = 0$, we get by Theorem \ref{TheoremBSaWBS} and Proposition \ref{PropWBS} that $A$ is a relatively weakly compact weak $\xi$-Banach-Saks set, and thus $A$ is a $\xi$-Banach-Saks set by Proposition \ref{PropEqBsWbs}. Now suppose that $A$ is a $\xi$-Banach-Saks set. Then by Proposition \ref{PropEqBsWbs} $A$ is a relatively weakly compact weak $\xi$-Banach-Saks set. Therefore, $\operatorname{wbs}_\xi^s(A) = 0$ by Proposition \ref{PropWBS} and Theorem \ref{TheoremQuantBS}. Hence, $\operatorname{bs}_\xi^s(A) = 0$ by Proposition \ref{PropCharBS}.
\end{proof}

\section{The quantities as functions of $\xi$} \label{section:QuantitiesAsFunctions}

In this section we will analyse the functions $\operatorname{bs}^s_\xi(A)$, $\operatorname{wbs}^s_\xi(A)$, $\operatorname{wus}_\xi(A)$ and $\operatorname{sm}_\xi (A)$ for a fixed bounded subset $A$ of a Banach space $X$ as functions of $\xi$. We begin with the quantities $\operatorname{wus}_\xi$ and $\operatorname{sm}_\xi$ and prove the simple observation that they are non-increasing.

\begin{lemma} \label{LemmaMonoWUSSM}
    Let $A$ be a bounded subset of a Banach space $X$ and let $\zeta < \xi < \omega_1$ be ordinals. Then $\operatorname{wus}_\xi (A) \leq \operatorname{wus}_\zeta (A)$ and $\operatorname{sm}_\xi (A) \leq \operatorname{sm}_\zeta (A)$.
\end{lemma}

\begin{proof}
    It follows from \cite[Lemma 2.1.8.(a)]{ArgyrosMercourakusTsarpalias} that there is $n = n(\zeta,\xi)$, such that for all $F \in \mathcal{S}_\zeta$ with $n \leq F$, we have $F \in \mathcal{S}_\xi$. In other words, if we set $N = \{m \in \mathbb{N}: \; n \leq m\}$, then $\mathcal{S}_\zeta [N] \subseteq \mathcal{S}_\xi$.
    
    Let $(x_n)_{n \in \mathbb{N}}$ be a sequence in $A$ which generates an $\ell_1^\xi$-spreading model with constant $c > 0$ then the sequence $(x_n)_{n \in N}$ generates an $\ell_1^\zeta$-spreading model with constant $c$, which gives us the inequality for the quantity $\operatorname{sm}$.
    
    Now, let $(x_n)_{n \in \mathbb{N}}$ be a sequence in $A$ which weakly converges to some $x \in X$ and is $(\xi,c)$-large for some $c > 0$. Then there is $M \in [\mathbb{N}]$ such that $\mathcal{S}_\xi^M \subseteq \mathcal{F}_c ((x_n-x)_{n \in \mathbb{N}})$. It is easy to check that this is equivalent to saying that $\mathcal{S}_\xi \subseteq \mathcal{F}_c ((x_n-x)_{n \in M})$. It follows that
    \begin{align*}
        \mathcal{S}_\zeta^N \subseteq \mathcal{S}_\zeta [N] \subseteq \mathcal{S}_\xi \subseteq \mathcal{F}_c((x_n-x)_{n \in M})
    \end{align*}
    and $(x_n)_{n \in M}$ is $(\zeta,c)$-large, which gives us the inequality for the quantity $\operatorname{wus}$.
\end{proof}

Now we turn our attention to the quantities $\operatorname{bs}^s_\xi$ and $\operatorname{wbs}^s_\xi$. We will first need the following definition.

\begin{definition}
    Let $(y_n)_{n \in \mathbb{N}}$ and $(z_n)_{n \in \mathbb{N}}$ be two sequences in a Banach space $X$. We say that the sequence $(z_n)_{n \in\mathbb{N}}$ is a \textit{non-increasing block convex combination} of the sequence $(y_n)_{n \in \mathbb{N}}$ if
        \begin{align*}
            z_n = \sum_{j=k_n + 1}^{k_{n+1}} \alpha (j) y_j,
        \end{align*}
    where $(k_n)_{n \in \mathbb{N}}$ is an increasing sequence of integers with $k_1 = 0$ and $(\alpha(j))_{j \in \mathbb{N}}$ is a non-increasing sequence of real numbers satisfying $\sum_{j=k_n + 1}^{k_{n+1}} \alpha(j) = 1$ for each $n \in \mathbb{N}$.
\end{definition}

For example, the $M$-summability method $([\xi+1]_n^M)_{n \in \mathbb{N}}$ is a non-increasing block convex combination of the $M$-summability method $(\xi_n^M)_{n \in \mathbb{N}}$ for any $\xi<\omega_1$ and $M \in [\mathbb{N}]$. It is readily proved that if a sequence $(x_n)_{n \in \mathbb{N}}$ is a non-increasing block convex combination of a sequence $(y_n)_{n \in \mathbb{N}}$, which is a non-increasing block convex combination of a sequence $(z_n)_{n \in \mathbb{N}}$, then $(x_n)_{n \in \mathbb{N}}$ is a non-increasing block convex combination of $(z_n)_{n \in \mathbb{N}}$.

Now we will prove an auxiliary lemma which shows that the quantity $\operatorname{cca}$ behaves well with respect to taking non-increasing block convex combinations.

\begin{lemma} \label{LemmaConvexBlock}
    Let $(y_n)_{n \in \mathbb{N}}$ and $(z_n)_{n \in \mathbb{N}}$ be two sequences in a Banach space $X$ such that $(z_n)_{n \in \mathbb{N}}$ is a non-increasing block convex combination of $(z_n)_{n \in \mathbb{N}}$. Then $\operatorname{cca}(z_n) \leq \operatorname{cca}(y_n)$.
\end{lemma}

\begin{proof}  
    Let $(k_n)_{n \in \mathbb{N}}$ and $(\alpha(j))_{j \in \mathbb{N}}$ be the sequences from the definition of non-increasing block convex combination. Let $c > 0$ and suppose that $\operatorname{cca}(y_n) \leq c$. Let us define $u_n = \frac{1}{n} \sum_{j=1}^{n} y_j$. The strategy is to show that the Ces\`aro means of the sequence $(z_n)_{n \in \mathbb{N}}$ can be written as convex combinations of $u_n$'s.
    
    Fix $\epsilon > 0$ and define $c' = c + \epsilon$. As $\operatorname{ca}(u_n) = \operatorname{cca}(y_n) < c'$, we can find $N_1 \in \mathbb{N}$ such that $\norm{u_j - u_i} \leq c'$ for all $i,j > N_1$. We define for $n \in \mathbb{N}$ and $j \leq k_{n+1}$
    \begin{align*}
        \beta_n (j) =
            \begin{cases}
                \alpha (k_{n+1}) k_{n+1}& \dots \; j = k_{n+1} \\
                \left( \alpha(j) - \alpha(j+1) \right) j & \dots \; j < k_{n+1}.
            \end{cases}
    \end{align*}
    Then we have for $n \in \mathbb{N}$
    \begin{align*}
        \frac{1}{n} \sum_{j=1}^n z_j &= \frac{1}{n} \sum_{j=1}^n  \sum_{i=k_j+1}^{k_{j+1}} \alpha (i) y_i = \frac{1}{n} \sum_{j=1}^{k_{n+1}} \alpha(j) y_j \\
        &= \frac{1}{n} \left( \alpha(k_{n+1}) \sum_{j=1}^{k_{n+1}} y_j + \sum_{j=1}^{k_n} (\alpha(j)-\alpha(j+1)) \sum_{i=1}^{j} y_i \right) \\
        &= \frac{1}{n} \left( \alpha(k_{n+1}) k_{n+1} u_{k_{n+1}} + \sum_{j=1}^{k_n} (\alpha(j)-\alpha(j+1)) j u_j \right) \\
        &= \frac{1}{n} \sum_{j=1}^{k_{n+1}} \beta_n(j) u_j.
    \end{align*}
    
    We will now prove by induction over $n$ that $\sum_{j=1}^{k_{n+1}} \beta_n (j) = n$. If $n = 1$, we have, since $k_1 = 0$,
    \begin{align*}
        \sum_{j=1}^{k_2} \beta_1(j) = \alpha(k_2) k_2 + \sum_{j = k_1 + 1}^{k_2 - 1} (\alpha(j) - \alpha(j+1)) j = \sum_{j=k_1 + 1}^{k_2} \alpha(j) = 1.
    \end{align*}
    Now suppose that for $n \in \mathbb{N}$ the equality $\sum_{j=1}^{k_{n+1}} \beta_n (j) = n$ holds. Notice that if $j < k_{n+1}$, we have $\beta_n(j) = \beta_{n+1}(j)$. Hence,
    \begin{align*}
        \sum_{j=1}^{k_{n+2}} \beta_{n+1}(j) - n &= \sum_{j=1}^{k_{n+2}} \beta_{n+1}(j) - \sum_{j=1}^{k_{n+1}} \beta_n(j) = \sum_{j=k_{n+1}}^{k_{n+2}} \beta_{n+1}(j) - \beta_n(k_{n+1}) \\
        &= \alpha(k_{n+2}) k_{n+2} + \sum_{j=k_{n+1}}^{k_{n+2}-1} (\alpha(j) - \alpha(j+1))j - \alpha(k_{n+1}) k_{n+1} \\
        &= \sum_{k_{n+1}+1}^{k_{n+2}} \alpha(j) = 1
    \end{align*}
    and the induction step follows.
    
    We proceed with estimating
    \begin{align*}
        \frac{1}{n} \sum_{j=1}^n z_j - \frac{1}{m} \sum_{i=1}^m z_i = \frac{1}{n} \sum_{j=1}^{k_{n+1}} \beta_n(j) u_j - \frac{1}{m} \sum_{i=1}^{k_{m+1}} \beta_m(i) u_i.
    \end{align*}
    Since $\sum_{j=1}^{k_{n+1}} \beta_n(j) = n$ and $\sum_{i=1}^{k_{m+1}} \beta_m(i) = m$, we have
    \begin{align*}
        \frac{1}{n} \sum_{j=1}^{k_{n+1}} \beta_n(j) u_j &- \frac{1}{m} \sum_{i=1}^{k_{m+1}} \beta_m(i) u_i = \frac{1}{nm} \sum_{j=1}^{k_{n+1}} \sum_{i=1}^{k_{m+1}} \beta_n(j) \beta_m(i) (u_j-u_i) \\
        = \frac{1}{nm} \Bigg( &\sum_{j=1}^{N_1} \sum_{i=1}^{k_{m+1}} \beta_n(j) \beta_m(i) (u_j-u_i) \\
        + &\sum_{j=N_1 + 1}^{k_{n+1}} \sum_{i=1}^{N_1} \beta_n(j) \beta_m(i) (u_j-u_i) \\
        + &\sum_{j=N_1+1}^{k_{n+1}} \sum_{i=N_1+1}^{k_{m+1}} \beta_n(j) \beta_m(i) (u_j-u_i) \Bigg).
    \end{align*}
    It follows from boundedness of the sequence $(y_n)_{n \in \mathbb{N}}$ that the sequence $(u_n)_{n \in \mathbb{N}}$ is also bounded. Let $M > 0$ be such that $\norm{u_n} \leq M$ for all $n \in \mathbb{N}$. We can find $N_2 > N_1$ such that for all $k > N_2$ we have $\frac{2M(N_1+1)}{k} < \epsilon$. Fix any $m,n > N_2$. Then
    \begin{align*}
        \frac{1}{nm} \sum_{j=1}^{N_1} \sum_{i=1}^{k_{m+1}} \beta_n(j) \beta_m(i) \norm{u_j-u_i} \leq \frac{2M(N_1 +1)m}{nm} &< \epsilon \\
        \frac{1}{nm} \sum_{j=N_1+1}^{k_{n+1}} \sum_{i=1}^{N_1} \beta_n(j) \beta_m(i) \norm{u_j-u_i} \leq \frac{2Mn(N_1+1)}{nm} &< \epsilon.
    \end{align*}
    The first inequalities on each line above hold as
    \begin{align*}
        \sum_{j=1}^{N_1} \beta_n(j) = \sum_{j=1}^{N_1} \beta_{N_1 + 1}(j) < \sum_{j=1}^{k_{N_1+1}} \beta_{N_1+1}(j) = N_1 + 1
    \end{align*}
    and analogically $\sum_{i=1}^{N_1} \beta_m(i) < N_1+1$.
    
    What is left is the estimate of the third term, which follows easily from the choice of $N_1$
    \begin{align*}
        \frac{1}{nm} \sum_{j= N_1+1}^{k_{n+1}} \sum_{i = N_1 + 1}^{k_{m+1}} \beta_n(j) \beta_m(i) \norm{u_j-u_i} \leq \frac{c'nm}{nm} = c'.
    \end{align*}
    
    We have thus shown that for $m,n > N_2$ we have
    \begin{align*}
        \norm{\frac{1}{n} \sum_{j=1}^n z_j - \frac{1}{m} \sum_{i=1}^m z_i} \leq 2 \epsilon + c' = 3 \epsilon + c.
    \end{align*}
    As $\epsilon > 0$ was arbitrary, we get $\operatorname{cca}(z_n) \leq c$.
\end{proof}

\begin{lemma}\label{LemmaNonincreasingBlockConvex}
    For every $\xi \leq \zeta < \omega_1$ and $M \in [\mathbb{N}]$ there is $N \in [M]$ such that the following statements hold:
    \begin{enumerate}[(a)]
        \item There is an increasing sequence of integers $(n_k)_{k \in \mathbb{N}}$ such that $N = \bigcup_{k=1}^\infty \operatorname{supp} \xi_{n_k}^M$.
        \item $(\zeta_j^N)_{j \in \mathbb{N}}$ is a non-increasing block convex combination of $(\xi_j^N)_{j \in \mathbb{N}}$.
    \end{enumerate}
\end{lemma}

\begin{proof}
    Let us abbreviate by $S(\xi,\zeta,M)$ the statement of the lemma for $\xi \leq \zeta$ and $M \in [\mathbb{N}]$. We will prove the lemma by induction over $\zeta$. Note that the statement $S(\xi,\zeta,M)$ is true if $\xi = \zeta$, just take $N = M$. Hence, we just need to prove the statements with strict inequality $\xi < \zeta$. If $\zeta = 0$, then the only possible choice for $\xi \leq \zeta$ is $\xi = \zeta = 0$ and we are done.
    
    Let $\zeta + 1 > 0$ be a successor ordinal and suppose that $S(\xi,\eta,M)$ holds for any $\xi \leq \eta < \zeta + 1$ and $M \in [\mathbb{N}]$. Fix $\xi < \zeta + 1$ and $M \in [\mathbb{N}]$. By the induction hypothesis the statement $S(\xi, \zeta, M)$ is valid. Let $N \in [M]$ be witnessing that. Then the property (a) of $S(\xi,\zeta+1,M)$ is the same as the property (a) of $S(\xi,\zeta,M)$, and so is satisfied. It follows from the definition of the $N$-summability method $([\zeta+1]_n^N)_{n \in \mathbb{N}}$ that $([\zeta + 1]_j^N)_{j \in \mathbb{N}}$ is a non-increasing block convex combination of $(\zeta_j^N)_{j \in \mathbb{N}}$. But $(\zeta_j^N)_{j \in \mathbb{N}}$ is in turn a non-increasing block convex combination of $(\xi_j^N)_{j \in \mathbb{N}}$. It follows that $([\zeta + 1]_j^N)_{j \in \mathbb{N}}$ is a non-increasing block convex combination of $(\xi_j^N)_{j \in \mathbb{N}}$ and the property (b) of $S(\xi,\zeta+1,M)$ also holds. Hence, the statement $S(\xi,\zeta+1,M)$ holds.
    
    Let $\zeta > 0$ be a limit ordinal and suppose that $S(\xi,\eta,M)$ holds for any $\xi \leq \eta < \zeta$ and $M \in [\mathbb{N}]$. Fix $\xi < \zeta$ and $M \in [\mathbb{N}]$. Let $(\zeta_n)_{n \in \mathbb{N}}$ be the sequence of successor ordinals increasing to $\zeta$ used to define the Schreier family $\mathcal{S}_\zeta$. Then $\xi < \zeta_{n_0}$ for some $n_0 \in \mathbb{N}$. Let $N_0 \in [M]$ be the set witnessing the validity of $S(\xi,\zeta_{n_0},M)$ and set $M_0 = N_0 \setminus \left(\operatorname{supp} [\zeta_{n_0}]_1^{N_0} \right)$. We proceed recursively: suppose that for $k \geq 0$ the set $M_k$ has already been defined. Set
    \begin{itemize}
        \item $N_{k+1}$ to be the set witnessing the validity of $S(\zeta_{n_0+k}, \zeta_{n_0+k+1}, M_k)$;
        \item $M_{k+1} = N_{k+1} \setminus \left(\operatorname{supp} [\zeta_{n_{0}+k+1}]_1^{N_{k+1}} \right)$.
    \end{itemize}
    Let
    \begin{align*}
        N = \bigcup_{k=0}^\infty \operatorname{supp} [\zeta_{n_0+k}]_1^{N_k} \hspace{1cm} \text{and} \hspace{1cm} P = N \cup \bigcup_{k=1}^{n_0-1} \operatorname{supp} \zeta_k^M. 
    \end{align*}
    By the definition of the $P$-summability method $(\zeta_k^P)_{k \in \mathbb{N}}$ we have $\zeta_k^P = [\zeta_k]^{P_k}_1$ where $P_1 = P$ and $P_{k+1} = P_k \setminus \operatorname{supp}[\zeta_k]_1^{P_k}$. By P.3. in \cite[p. 171]{ArgyrosMercourakusTsarpalias} we get that $\zeta_k^P = \zeta_k^M$ for $k=1,\dots,n_0-1$. It follows that $\operatorname{supp}[\zeta_{n_0+k}]_1^{N_k}$ is an initial segment of $P_{n_0+k}$ for $k \geq 0$. Hence, again by P.3. in \cite[p. 171]{ArgyrosMercourakusTsarpalias}, we get $[\zeta_{n_0+k}]_1^{P_{n_0+k}} = [\zeta_{n_0+k}]_1^{N_k}$ for $k \geq 0$. Now we can use P.4. in \cite[p. 171]{ArgyrosMercourakusTsarpalias} and the fact that
    \begin{align*}
        N = \bigcup_{k=0}^\infty \operatorname{supp}[\zeta_{n_0+k}]_1^{N_k} = \bigcup_{k=0}^\infty \operatorname{supp}[\zeta_{n_0+k}]_1^{P_{n_0+k}} = \bigcup_{k=0}^\infty \operatorname{supp} \zeta_{n_0+k}^{P} = \bigcup_{k=n_0}^\infty \operatorname{supp} \zeta_{k}^{P} 
    \end{align*}
    to conclude that for $k \geq 0$
    \begin{align*}
        \zeta_{k+1}^N = \zeta_{n_0+k}^P = [\zeta_{n_0+k}]_1^{P_{n_0+k}} = [\zeta_{n_0+k}]_1^{N_k}.
    \end{align*}
    We will now show that $N \in [M]$ witnesses the validity of $S(\xi,\zeta,M)$.
    
    First, let us prove by induction that for each $n \geq 0$ the sequence $([\zeta_{n_0+n}]_j^{N_n})_{j \in \mathbb{N}}$ is a non-increasing block convex combination of $(\xi_j^{N_n})_{j \in \mathbb{N}}$. The case $n = 0$ follows immediately from the choice of $N_0$ and property (b) of $S(\xi,\zeta_{n_0},M)$. Suppose the claim holds for some $n \geq 0$. By the choice of $N_{n+1}$ as the set witnessing $S(\zeta_{n_0+n}, \zeta_{n_0+n+1}, M_{n})$, we can use property (b) to get that $([\zeta_{n_0+n+1}]_j^{N_{n+1}})_{j \in \mathbb{N}}$ is a non-increasing block convex combination of $([\zeta_{n_0+n}]_j^{N_{n+1}})_{j \in \mathbb{N}}$. But by the induction hypothesis $([\zeta_{n_0+n}]_j^{N_{n}})_{j \in \mathbb{N}}$ is a non-increasing block convex combination of $(\xi_j^{N_n})_{j \in \mathbb{N}}$, and hence, by property (a) of $S(\zeta_{n_0+n}, \zeta_{n_0+n+1}, M_{n})$ and P.4. in \cite[p. 171]{ArgyrosMercourakusTsarpalias}, also $([\zeta_{n_0+n}]_j^{N_{n+1}})_{j \in \mathbb{N}}$ is a non-increasing block convex combination of $(\xi_j^{N_{n+1}})_{j \in \mathbb{N}}$. Thus $([\zeta_{n_0+n+1}]_j^{N_{n+1}})_{j \in \mathbb{N}}$ is a non-increasing block convex combination of a non-increasing block convex combination of $(\xi_j^{N_{n+1}})_{j \in \mathbb{N}}$, and hence is itself a non-increasing block convex combination of $(\xi_j^{N_{n+1}})_{j \in \mathbb{N}}$. Therefore the claim is proved.
    
    It follows that for each $n \in \mathbb{N}$ we have
    \begin{align*}
        [\zeta_{n_0+n-1}]_j^{N_{n-1}} = \sum_{i=k^n_j+1}^{k^n_{j+1}} \alpha_n(i) \xi_i^{N_{n-1}},
    \end{align*}
    where $(k^n_j)_{j \in \mathbb{N}}$ is an increasing sequence of integers with $k^n_1 = 0$ and $(\alpha_n(i))_{i\in \mathbb{N}}$ is the sequence of coefficients of non-increasing block convex combinations.
    
    Let us recursively define an increasing sequence of integers $(k_n)_{n\in \mathbb{N}}$ and a sequence of positive numbers $(\alpha(j))_{j \in \mathbb{N}}$ satisfying $\sum_{j=k_n+1}^{k_{n+1}} \alpha(j) = 1$ for each $n \in \mathbb{N}$. Set $k_1 = 0$, $k_2 = k_2^1$ and $\alpha(j) = \alpha_1(j)$ for $1 \leq j \leq k_2^1$. If for some $n \in \mathbb{N}$ the number $k_n$ has already been defined, set $k_{n+1} = k_n + k_2^n$ and for $k_n + 1 \leq j \leq k_{n+1}$ set $\alpha(j) = \alpha_{n+1}(j - k_n)$. We will also need the fact that
    \begin{align*}
        \xi_j^{N_{n-1}} = \xi_{k_n+j}^N
    \end{align*}
    which is readily proved by induction over $j$ using P.3. and P.4. in \cite[p. 171]{ArgyrosMercourakusTsarpalias} and the fact that
    \begin{align*}
        N = \bigcup_{n=1}^\infty \operatorname{supp}[\zeta_{n_0+n-1}]_1^{N_{n-1}} = \bigcup_{n=1}^\infty \bigcup_{j=1}^{k_2^n} \operatorname{supp} \xi_j^{N_{n-1}}.
    \end{align*}
    
    We will now show that the sequence $(\alpha(j))_{j \in \mathbb{N}}$ in non-increasing. The only part that does not follow from the choice of the sequences $(\alpha_n(j))_{j=1}^{\infty}$ is that $\alpha(k_{n+1}) \geq \alpha(k_{n+1} + 1)$, that is $\alpha_n(k_2^n) \geq \alpha_{n+1}(1)$, for every $n \in \mathbb{N}$. This follows from the property $S(\zeta_{n_0+n-1}, \zeta_{n_0+n},M_{n-1})$. Indeed, by property (a) and P.4. in \cite[p. 171]{ArgyrosMercourakusTsarpalias} we have that for each $k \in \mathbb{N}$
    \begin{align*}
        [\zeta_{n_0+n-1}]_k^{N_n} = [\zeta_{n_0+n-1}]_{n_k}^{N_{n-1}}
    \end{align*}
    for some increasing sequence of integers $(n_k)_{k \in \mathbb{N}}$. Further, by property (b) of $S(\zeta_{n_0+n-1}, \zeta_{n_0+n},M_{n-1})$ we have
    \begin{align*}
        \zeta_{n+1}^N &= [\zeta_{n_0+n}]_1^{N_n} = \sum_{j=1}^m \beta_j [\zeta_{n_0+n-1}]_j^{N_n} = \sum_{j=1}^m \beta_j [\zeta_{n_0+n-1}]_{n_j}^{N_{n-1}} \\
        &= \sum_{j=1}^m \beta_j \sum_{i=k_{n_j}^{n}+1}^{k_{n_j+1}^{n}} \alpha_{n}(i) \xi_i^{N_{n-1}} \\
        &= \sum_{j=1}^m \sum_{i=k_{n_j}^{n}+1}^{k_{n_j+1}^{n}} \left( \beta_j \alpha_{n}(i) \right) \xi_{k_n+i}^N.
    \end{align*}
    for some $m \in \mathbb{N}$ and a non-increasing sequence $(\beta_j)_{j=1}^m$ satisfying $\beta_1 \leq 1$. As we also have
    \begin{align*}
        \zeta_{n+1}^N = [\zeta_{n_0+n}]_1^{N_n} = \sum_{i=k^{n+1}_1 + 1}^{k_{2}^{n+1}} \alpha_{n+1}(i) \xi_i^{N_n} = \sum_{i=k^{n+1}_1 + 1}^{k_{2}^{n+1}} \alpha_{n+1}(i) \xi_{k_{n+1}+i}^N
    \end{align*}
    and $\xi^N_{j}$, $j \in \mathbb{N}$, have disjoint supports, we get
    \begin{align*}
        \alpha_{n+1}(1) = \beta_1 \alpha_{n}(k_{n_1}^{n}+1) \leq \alpha_{n}(k_{n_1}^{n}+1) \leq \alpha_n(k^n_2),
    \end{align*}
    where the last inequality holds as $n_1 \geq 2$, which in turn follows from the choice of $M_{n-1} = N_{n-1} \setminus \left(\operatorname{supp} [\zeta_{n_0+n-1}]_1^{N_{n-1}} \right)$ -- the set $N_n \in [M_{n-1}]$ cannot contain $\operatorname{supp} [\zeta_{n_0+n-1}]_1^{N_{n-1}}$, and the fact that the sequence $(\alpha_n(j))_{j \in \mathbb{N}}$ in non-increasing.
    
    Hence, for any $n \in \mathbb{N}$
    \begin{align*}
        \zeta_n^N = [\zeta_{n_0+n-1}]_1^{N_{n-1}} = \sum_{j=k_n+1}^{k_{n+1}} \alpha(j) \xi_j^N
    \end{align*}
    and property (b) of $S(\xi,\zeta,M)$ is valid for $N$. Property (a) is also valid as we have already shown:
    \begin{align*}
        N = \bigcup_{n=1}^\infty \operatorname{supp}[\zeta_{n_0+n-1}]_1^{N_{n-1}} = \bigcup_{n=1}^\infty \bigcup_{j=1}^{k_2^n} \operatorname{supp} \xi_j^{N_{n-1}} = \bigcup_{n=1}^\infty \bigcup_{j=1}^{k_2^n} \operatorname{supp} \xi_{k_n+j}^N.
    \end{align*}
    Hence, the induction step for limit ordinals is done and the lemma is proved.
\end{proof}

\begin{prop}\label{PropMonoWBS}
    Let $(x_n)_{n \in \mathbb{N}}$ be a bounded sequence in a Banach space $X$ and $\xi< \zeta < \omega_1$. Then $\widetilde{\operatorname{cca}}_\zeta^s((x_n)_{n \in \mathbb{N}}) \leq \widetilde{\operatorname{cca}}_\xi^s ((x_n)_{n \in \mathbb{N}})$. In particular, for any bounded subset $A$ of $X$ we have $\operatorname{wbs}^s_\zeta(A) \leq \operatorname{wbs}^s_\xi (A)$ and $\operatorname{bs}^s_\zeta(A) \leq \operatorname{bs}^s_\xi(A)$.
\end{prop}

\begin{proof}
    Let $\widetilde{\operatorname{cca}}_\xi^s((x_n)_{n \in \mathbb{N}}) < c$ for some $c > 0$. Then for every $M \in [\mathbb{N}]$ there is $N \in [M]$ such that $\operatorname{cca}\left( \zeta_n^N \cdot (x_k)_{k \in \mathbb{N}} \right) < c$. We will show using the infinite Ramsey theorem \cite[Theorem 10.1.3.]{nigel2006topics} that this implies that for every $M \in [\mathbb{N}]$ there is $N \in [M]$ such that for all $L \in [N]$ we have $\operatorname{cca}\left( \zeta_n^L \cdot (x_k)_{k \in \mathbb{N}} \right) < c$. Fix any $M \in [\mathbb{N}]$ and define
    \begin{align*}
        A_1 &= \left\{P \in [M]: \; \operatorname{cca}\left( \zeta_n^P \cdot (x_k)_{k \in \mathbb{N}} \right) < c \right\} \\
        A_2 &= [M]  \setminus A_1.
    \end{align*}
    The set $A_1$ is Ramsey. Indeed, for $P \in [M]$ we have that $\operatorname{cca} \left( \zeta_n^P \cdot (x_k)_{k \in \mathbb{N}} \right) < c$ if and only if
    \begin{align*}
        \exists m \in \mathbb{N} \; \exists n \in \mathbb{N} \; \forall i \geq m \; \forall j \geq m: \; \norm{\frac{1}{i} \sum_{l = 1}^i \zeta_l^P \cdot (x_k)_{k \in \mathbb{N}} - \frac{1}{j} \sum_{l = 1}^j \zeta_l^P \cdot (x_k)_{k \in \mathbb{N}}} \leq c - \frac{1}{n}
    \end{align*}
    and for any $i,j,n \in \mathbb{N}$ the set
    \begin{align*}
        A(i,j,n) = \left\{ P \in [M]: \; \norm{\frac{1}{i} \sum_{l = 1}^i \zeta_l^P \cdot (x_k)_{k \in \mathbb{N}} - \frac{1}{j} \sum_{l = 1}^j \zeta_l^P \cdot (x_k)_{k \in \mathbb{N}}} \leq c - \frac{1}{n} \right\}
    \end{align*}
    is open by P.3. in \cite[p. 171]{ArgyrosMercourakusTsarpalias}. Hence,
    \begin{align*}
        A_1 = \bigcup_{m \in \mathbb{N}} \bigcup_{n \in \mathbb{N}} \bigcap_{i \geq m} \bigcap_{j \geq m} A(i,j,n)
    \end{align*}
    is Borel, and thus Ramsey. It follows from the infinite Ramsey theorem that there is $N \in [M]$ such that either $[N] \subseteq A_1$ or $[N] \subseteq A_2$. But we have already seen that the latter case is impossible. Hence, $[N] \subseteq A_1$ which is precisely what we wanted to show.
    
    It follows from Lemmata \ref{LemmaConvexBlock} and \ref{LemmaNonincreasingBlockConvex} that there is $L \in [N] \subseteq [M]$ such that $\operatorname{cca}\left( \zeta_n^L \cdot (x_k)_{k \in \mathbb{N}} \right) \leq \operatorname{cca}\left( \xi_n^L \cdot (x_k)_{k \in \mathbb{N}} \right) \leq c$. Therefore, we have found for every $M \in [\mathbb{N}]$ some $L \in [M]$ such that $\operatorname{cca}\left( \zeta_n^L \cdot (x_k)_{k \in \mathbb{N}} \right) \leq c$, which implies the desired inequality $\widetilde{\operatorname{cca}}_\zeta^s ((x_n)_{n \in \mathbb{N}}) \leq c$.
\end{proof}

It follows from Proposition \ref{PropMonoWBS} that the quantities $\operatorname{bs}^s_\xi$ and $\operatorname{wbs}^s_\xi$ are non-increasing with respect to $\xi$. It is unclear if the same holds for the quantities $\operatorname{bs}_\xi$ and $\operatorname{wbs}_\xi$. We do, however, have monotony if $\zeta$ is a finite successor of $\xi$.

\begin{lemma}
    Let $\xi < \omega_1$ and $\zeta = \xi + l$ for some $l \in \mathbb{N}$. Let $(x_n)_{n \in \mathbb{N}}$ be a bounded sequence in a Banach space $X$ and $M \in [\mathbb{N}]$. Then $\operatorname{cca}(\zeta_n^M \cdot (x_k)_{k \in \mathbb{N}}) \leq \operatorname{cca}(\xi_n^M \cdot (x_k)_{k \in \mathbb{N}})$. In particular, $\widetilde{\operatorname{cca}}_\zeta ((x_n)_{n \in \mathbb{N}}) \leq \widetilde{\operatorname{cca}}_\xi ((x_n)_{n \in \mathbb{N}})$, and for any bounded subset $A$ of $X$ we have $\operatorname{bs}_\zeta(A) \leq \operatorname{bs}_\xi(A)$ and $\operatorname{wbs}_\zeta(A) \leq \operatorname{wbs}_\xi(A)$.
\end{lemma}

\begin{proof}
    This follows easily by induction over $l \in \mathbb{N}$ and the fact that for any $M \in [\mathbb{N}]$ and $l \in \mathbb{N} \cup \{0\}$ the $M$-summability method $([\zeta + l + 1]_n^M)_{n \in \mathbb{N}}$ is a non-increasing block convex combination of the $M$-summability method $([\zeta + l]_n^M)_{n \in \mathbb{N}}$. Therefore, we just need to invoke Lemma \ref{LemmaConvexBlock}.
\end{proof}

We define another quantity for a bounded subset $A$ of a Banach space $X$.

\begin{definition}
Let $A$ be a bounded subset of a Banach space $X$. We define
    \begin{align*}
        \delta_0(A) = \min_{\xi < \omega_1} \operatorname{bs}^s_\xi(A).
    \end{align*}
\end{definition}

This quantity $\delta_0$ is a measure of weak non-compactness for separable sets. To prove this we will first need the following lemma. Notice that the assumptions of Lemma \ref{LemmaBourgain} cannot be met; it is only used to prove a contradiction in the proof of Proposition \ref{PropDeltaSeparable}.

\begin{lemma} \label{LemmaBourgain}
    Let $A$ be a bounded separable relatively weakly compact subset of a Banach space $X$ which satisfies $\delta_0(A) > 0$. Then the canonical basis of $\ell_1$ embeds into $\overline{A}$.
\end{lemma}

\begin{proof}
    We can suppose that $A \subseteq B_X$. Let $c>0$ be such that $\delta_0(A) > 4c$. We define a tree $\mathcal{T}$ on $X$ as
    \begin{align*}
        \mathcal{T} = \left\{(x_1,\dots,x_n) \in \overline{A}^n: \; \norm{\sum_{j=1}^n a_j x_j} \geq c \sum_{j=1}^n |a_j| \; \text{for all } (a_j)_{j=1}^n \in \mathbb{R}^n \right\}.
    \end{align*}
    Note that this is a modification of the tree $\mathcal{T}(X,c)$, used by Bourgain \cite{Bourgain1980} to define the $\ell_1$-index, that is made only of sequences in $\overline{A}$ instead of $B_X$. We will further use the terminology from \cite{Bourgain1980}. If we can show that $\mathcal{T}$ is ill-founded, any infinite branch of $\mathcal{T}$ can serve as an isomorphic copy of the canonical basis of $\ell_1$ and we are done. As $\mathcal{T}$ is obviously a closed tree, it is enough to show that the order of $\mathcal{T}$ is equal to $\omega_1$ and invoke \cite[Proposition 10]{Bourgain1980}.
    
    Fix $\xi < \omega_1$. As $\operatorname{bs}^s_\xi(A) > 4c$, we can use Proposition \ref{PropCharBS} and Theorem \ref{TheoremQuantBS} to find a sequence $(x_n)_{n \in \mathbb{N}}$ in $A$ which generates an $\ell_1^{\xi+1}$-spreading model with constant $c$. This implies that
    \begin{align*}
        \left\{ (x_n)_{n \in F}: \; F \in \mathcal{S}_{\xi+1} \right\} \subseteq \mathcal{T}.
    \end{align*}
    It follows from \cite[Lemma 4.10.]{AlspachArgyros} that the order of $\mathcal{S}_{\xi+1}$ (as a tree on $\mathbb{N}$) is equal to $\omega^{\xi+1}$. It is not hard to see that this implies that the order of $\mathcal{T}$ is at least $\omega^{\xi+1}$. But $\xi<\omega_1$ was arbitrary, and hence the order of $\mathcal{T}$ is $\omega_1$.
\end{proof}

\begin{prop} \label{PropDeltaSeparable}
    Let $A$ be a bounded separable subset of a Banach space $X$. Then $\delta_0(A) = 0$ if and only if $A$ is relatively weakly compact.
\end{prop}

\begin{proof}
    If $A$ is not relatively weakly compact, then $0 < \operatorname{wck}_X(A) \leq \operatorname{bs}^s_\xi(A)$ for all $\xi<\omega_1$ by the virtue of Theorem \ref{TheoremBSaWBS}, and therefore $\delta_0(A) > 0$.
    
    On the other hand, let $A$ be relatively weakly compact. Let us assume for a contradiction that $\delta_0(A) > 0$. It follows from Lemma \ref{LemmaBourgain} that $\overline{A}$ contains a sequence equivalent to the canonical basis of $\ell_1$ which contradicts the relative weak compactness of $A$. Hence, $\delta_0(A) = 0$ and we are done.
\end{proof}

Note that separability of $A$, was essential in the proof of the preceding theorem, as the result of Bourgain \cite{Bourgain1980} (Lemma \ref{LemmaBourgain}) relies on an argument based on trees which is valid only in separable spaces. We will illustrate the necessity of separability for $\delta_0$ to be a measure of weak non-compactness in Example \ref{ExampleDeltaNonseparable} below. However, the quantity $\delta_0$ can be modified to be a measure of weak non-compactness.

\begin{definition}
    Let $A$ be a bounded subset of a Banach space $X$. We define
    \begin{align*}
        \delta(A) = \sup \{\delta_0(B): \; B \subseteq A \text{ separable}\}.
    \end{align*}
\end{definition}

\begin{prop} \label{PropDelta}
    Let $A$ be a bounded subset of a Banach space $X$. Then $\delta(A) = 0$ if and only if $A$ is relatively weakly compact.
\end{prop}

\begin{proof}
    It follows from the Eberlein-Šmulyan theorem that $A$ is relatively weakly compact if and only if each separable (or even countable) subset a $A$ is relatively weakly compact. Hence, the proposition follows from Proposition \ref{PropDeltaSeparable}.
\end{proof}

\section{Examples} \label{section:Examples}

In this section we investigate, whether the inequalities of Theorem \ref{TheoremQuantBS} and Theorem \ref{TheoremBSaWBS} are optimal and whether they can be strict. We begin with Theorem \ref{TheoremBSaWBS}, which stated that for any $\xi < \omega_1$ and any bounded set $A$ in some Banach space $X$ we have
\begin{align*}
    \max \{\operatorname{wck}_X(A),\operatorname{wbs}_\xi(A)\} \leq \operatorname{bs}_\xi(A) \leq \operatorname{bs}_\xi^s(A) \leq \beta(A).
\end{align*}

We will look at the following examples of classical spaces:

\begin{itemize}
    \item If $A = B_{C[0,1]}$, then $\operatorname{wbs}_{\xi} (A) = \beta(A) = 2$ as the space $C[0,1]$ contains the Schreier space of order $\xi$, see Example \ref{ExampleXiSchreier} below (in fact, it contains any separable Banach space). Hence,
    \begin{align*}
        \max \{\operatorname{wck}_{C[0,1]}(A),\operatorname{wbs}_\xi(A)\} = \operatorname{bs}_\xi(A) = \operatorname{bs}^s_\xi(A) = \beta(A).
    \end{align*}
    \item If $A = B_{\ell_1}$, then $\operatorname{wbs}_\xi (A) = 0$ as there are no nontrivial weakly null sequences in $\ell_1$. Further, $\operatorname{wck}_{\ell_1}(A) = 1$, as $\ell_1$ is not reflexive, and $\operatorname{bs}_{\xi}(A) = \beta(A) = 2$ (the fact that $\operatorname{bs}_{\xi} (A) = 2$ is witnessed by the canonical basis and $\operatorname{bs}_{\xi}(A) \leq \beta (A) \leq 2$ by Theorem \ref{TheoremBSaWBS} and the triangle inequality). Hence,
    \begin{align*}
        \max \{\operatorname{wck}_{\ell_1}(A),\operatorname{wbs}_\xi(A)\} < \operatorname{bs}_\xi(A) = \operatorname{bs}_\xi^s(A) = \beta(A).
    \end{align*}
    \item If $A = B_{c_0}$, then $\operatorname{wbs}_\xi(A) = 0$ as $c_0$ has the weak Banach-Saks property, and thus also the weak $\xi$-Banach-Saks property, by \cite{Farnum1974}. Further, $\operatorname{wck}_{c_0}(A) = 1$, as $c_0$ is not reflexive, and $\beta(A) = 2$, as witnessed by the sequence $x_n = e_1 + \cdots + e_n - e_{n-1}$. The quantity $\operatorname{bs}_\xi(A)$ is harder to compute. It follows from \cite[Theorem 5.2.]{Bendov__2015} that $\operatorname{bs}_0 (A) = \operatorname{bs}_0^s(A) \leq 1$. Hence, by Proposition \ref{PropMonoWBS} we have $\operatorname{bs}_\xi(A) \leq \operatorname{bs}_\xi^s(A) \leq \operatorname{bs}_0^s(A) \leq 1$. On the other hand $\operatorname{bs}^s_\xi(A) \geq \operatorname{bs}_\xi(A) \geq \operatorname{wck}_{c_0}(A) = 1$, and therefore
    \begin{align*}
        \max \{\operatorname{wck}_{c_0}(A),\operatorname{wbs}_\xi(A)\} = \operatorname{bs}_\xi(A) = \operatorname{bs}^s_\xi(A) < \beta(A).
    \end{align*}
\end{itemize}

So, the inequalities of Theorem \ref{TheoremBSaWBS} are optimal and, possibly except the inequality $\operatorname{bs}_\xi(A) \leq \operatorname{bs}_\xi^s(A)$, can be strict. We proceed with Theorem \ref{TheoremQuantBS}, which stated that for any $\xi < \omega_1$ and any bounded subset $A$ of some Banach space $X$ we have

\begin{align*}
    2\operatorname{sm}_{\xi+1} (A) \leq \operatorname{wbs}_\xi (A) \leq \operatorname{wbs}_\xi^s(A) \leq 2 \operatorname{wus}_{\xi+1} (A) \leq 4 \operatorname{sm}_{\xi+1} (A).
\end{align*}

\begin{example} \label{ExampleXiSchreier}
    Let $\xi < \omega_1$ and $X_\xi$ denote the Schreier space of order $\xi$, that is the completion of $c_{00}$ under the norm
    \begin{align*}
        \norm{x} = \sup_{F \in \mathcal{S}_\xi} \norm{x \restriction F}_{\ell_1}.
    \end{align*}
    Where $x \restriction F$ denotes the sequence $(y_i)_{i \in \mathbb{N}}$ where $y_i = x_i$ for $i \in F$ and $y_i = 0$ otherwise. It can be shown using classical methods that the canonical sequence $(e_n)_{n \in \mathbb{N}}$ of $c_{00}$ is a normalized 1-unconditional basis of $X_\xi$. Further, the Bourgain's $\ell_1$-index of $X_\xi$ is countable (see \cite[Remmark 5.21.]{JuddOdell1998}), and hence $X_\xi$ does not contain $\ell_1$ by the result of Bourgain \cite{Bourgain1980}. Therefore, the basis $(e_n)_{n \in \mathbb{N}}$ is shrinking (see e.g. \cite[Theorem 3.3.1.]{nigel2006topics}) and in particular weakly null.
    
    Now let us consider $A = \{e_n: \; n \in \mathbb{N}\}$ as a bounded subset of $X_{\xi+1}$. We will show that
    \begin{enumerate}[(i)]
        \item $\operatorname{sm}_{\xi+1} (A) = 1$,
        \item $\operatorname{wus}_{\xi+1} (A) = 1$,
        \item $\operatorname{wbs}_\xi(A) = \operatorname{wbs}_\xi^s(A) = 2$.
    \end{enumerate}
    
    For any $F \in \mathcal{S}_{\xi+1}$ and $(a_n)_{n \in F} \in \mathbb{R}^F$ we have
    \begin{align*}
        \norm{\sum_{n \in F} a_n e_n} \geq \sum_{n \in F} |a_n|
    \end{align*}
    by the very definition of the norm of $X_{\xi+1}$. On the other hand, as $A$ is a subset of $B_{X_{\xi+1}}$, we get that $\operatorname{sm}_{\xi+1}(A) \leq 1$ by the triangle inequality. Hence, (i) is proved.
    
    We again notice that $A \subseteq B_{X_{\xi+1}}$, and thus $\operatorname{wus}_{\xi+1}(A) \leq 1$. On the other hand, we will show that for any $0<c<1$ we have $\mathcal{S}_{\xi+1} \subseteq \mathcal{F}_c = \mathcal{F}_c((e_n)_{n \in \mathbb{N}})$. Take any $F = (n_1,\dots,n_k) \in \mathcal{S}_{\xi+1}$ and define $x^* = e_{n_1}^* + \cdots + e_{n_k}^*$. Then for any $x = (x_n)_{n \in \mathbb{N}} \in X_{\xi+1}$ we have
    \begin{align*}
        |x^*(x)| = \left| \sum_{j \in F} x_j \right| \leq \sum_{j \in F} |x_j| = \norm{x \restriction F}_{\ell_1} \leq \norm{x}.
    \end{align*}
    Hence, $x^* \in B_{X_{\xi+1}^*}$. It follows, as  $x^*(e_{n_j}) = 1$ for $j = 1,\dots,k$, that $F \in \mathcal{F}_c$. We have proved that $(e_n)_{n \in \mathbb{N}}$ is $(\xi+1,c)$-large for any $0<c<1$, and thus that $\operatorname{wus}_{\xi+1}(A) \geq 1$. Therefore, (ii) is proved.
    
    (iii) now easily follows from Theorem \ref{TheoremQuantBS}.
\end{example}

\begin{example} \label{ExampleXiSchreierStar}
    Let $\xi < \omega_1$. We will consider an equivalent norm on the Schreier space $X_{\xi}$ of order $\xi$, namely
    \begin{align*}
        \norm{x}_* = \max \left\{\norm{x^+}, \norm{x^-} \right\},
    \end{align*}
    where $\norm{\cdot}$ is the norm defined in Example \ref{ExampleXiSchreier} and $x^\pm = (x_n^\pm)_{n \in \mathbb{N}}$ for $x = (x_n)_{n \in \mathbb{N}}$. Then $\norm{x}_* \leq \norm{x} \leq 2 \norm{x}_*$ for each $x \in X_{\xi}$ and $\norm{y}_* = \norm{y}$ for all $y$ in the positive cone of $X_{\xi}$ (that is $y$ with non-negative coordinates). In particular, $(e_n)_{n\in \mathbb{N}}$ is a weakly null normalized sequence in $(X_{\xi},\norm{\cdot}_*)$. Consider again $A = \{e_n: \; n \in \mathbb{N}\}$ as a bounded subset of $(X_{\xi+1},\norm{\cdot}_*)$. We will show the following:
    \begin{enumerate}[(i)]
        \item $\operatorname{sm}_{\xi+1} (A) = \frac{1}{2}$,
        \item $\operatorname{wus}_{\xi+1} (A) = 1$,
        \item $\operatorname{wbs}_\xi(A) = \operatorname{wbs}_\xi^s(A) = 1$.
    \end{enumerate}
    
    Fix any $F \in \mathcal{S}_{\xi+1}$ and $(a_n)_{n \in F} \in \mathbb{R}^F$. Then
    \begin{align*}
        \norm{\sum_{n \in F} a_n^+ e_n} \geq \sum_{n \in F} a_n^+ \hspace{1cm} \text{and} \hspace{1cm} \norm{\sum_{n \in F} a_n^- e_n} \geq \sum_{n \in F} a_n^-,
    \end{align*}
    as $F \in \mathcal{S}_{\xi+1}$. But then
    \begin{align*}
        \norm{\sum_{n \in F} a_n e_n}_* &= \max \left\{ \norm{\sum_{n \in F} a_n^+ e_n}, \norm{\sum_{n \in F} a_n^- e_n} \right\} \\ &\geq \max \left\{\sum_{n \in F} a_n^+, \sum_{n \in F} a_n^- \right\} \geq \frac{1}{2} \sum_{n \in F} |a_n|.
    \end{align*}
    Hence, $\operatorname{sm}_{\xi+1} (A) \geq \frac{1}{2}$. To show the other inequality it is enough to show that $\operatorname{sm}_1(A) \leq \frac{1}{2}$ and use the monotony provided by Lemma \ref{LemmaMonoWUSSM}. Let us have an arbitrary sequence $(f_n)_{n \in \mathbb{N}}$ in $A$. Note that the set $F = \{2,3\}$ belongs to the Schreier family $\mathcal{S}_1$. We define $(a_k)_{k \in F} \in \mathbb{R}^F$ by setting $a_2 = 1$, $a_3 = -1$. If $f_2 = f_3$, then
    \begin{align*}
        \norm{\sum_{k \in F} a_k f_k}_* = \norm{f_2 - f_3}_* =  0 \hspace{1cm} \text{but} \hspace{1cm} \sum_{k \in F} |a_k| = 2
    \end{align*}
    and $(f_n)_{n \in \mathbb{N}}$ cannot generate an $\ell_1^1$-spreading model. If $f_2 \neq f_3$, then
    \begin{align*}
        \norm{\sum_{k \in F} a_k f_k}_* = \norm{f_2 - f_3}_* =  1 \hspace{1cm} \text{but} \hspace{1cm} \sum_{k \in F} |a_k| = 2
    \end{align*}
    and $(f_n)_{n \in \mathbb{N}}$ cannot generate an $\ell_1^1$-spreading model with constant greater than $\frac{1}{2}$. In any case, we have shown that $\operatorname{sm}_1(A) \leq \frac{1}{2}$ and (i) is proved.
    
    Now we proceed with (ii). First we notice that $A \subseteq B_{X_{\xi+1}}$, and thus $\operatorname{wus}_{\xi+1}(A) \leq 1$. On the other hand, we will show that for $0<c<1$ we have $\mathcal{S}_{\xi+1} \subseteq \mathcal{F}_c = \mathcal{F}_c((e_n)_{n \in \mathbb{N}})$. Take any $F = (n_1,\dots,n_k) \in \mathcal{S}_{\xi+1}$ and define $x^* = e_{n_1}^* + \cdots + e_{n_k}^*$. Then for any $x = (x_j)_{j \in \mathbb{N}} \in X_{\xi+1}$ we have
    \begin{align*}
        |x^*(x)| = \left| \sum_{j \in F} x_j \right| \leq \max \left\{ \sum_{j \in F} x_j^+, \sum_{j \in F} x_j^- \right\} \leq \max \{\norm{x^+}, \norm{x^-}\} = \norm{x}_*.
    \end{align*}
    Hence, $x^* \in B_{X_{\xi+1}^*}$. But $x^*(e_{n_j}) = 1 > c$ for $j = 1,\dots,k$, and thus $F \in \mathcal{F}_c$. We have shown that $(e_n)_{n \in \mathbb{N}}$ is $(\xi+1,c)$-large for any $0<c<1$, which implies that $\operatorname{wus}_{\xi+1}(A) \geq 1$. But then $\operatorname{wus}_{\xi+1}(A) = 1$ and (ii) is proved.
    
    Finally, we prove (iii). It follows from (i) and Theorem \ref{TheoremQuantBS} that $\operatorname{wbs}_\xi(A) \geq 1$. The inequality $\operatorname{wbs}_\xi^s(A) \leq 1$ follows from the fact that for any sequence $(x_n)_{n \in \mathbb{N}}$ in $A$, any $N \in [\mathbb{N}]$ and any $k<l \in \mathbb{N}$ we have
    \begin{align*}
        &\norm{\frac{1}{k} \sum_{j=1}^k \xi_j^N \cdot (x_n)_{n \in \mathbb{N}} - \frac{1}{l} \sum_{j=1}^l \xi_j^N \cdot (x_n)_{n \in \mathbb{N}}}_* \\
        &= \max \left\{ \norm{\left(\frac{1}{k} - \frac{1}{l}\right) \sum_{j=1}^k \xi_j^N \cdot (x_n)_{n \in \mathbb{N}}}, \norm{\frac{1}{l} \sum_{j=k+1}^l \xi_j^N \cdot (x_n)_{n \in \mathbb{N}}} \right\} \leq 1,
    \end{align*}
    where the first equality holds as the summability methods $(\xi_j^N)_{j \in \mathbb{N}}$ have non-negative coefficients and the last inequality follows from the triangle inequality.
\end{example}

It follows from Example \ref{ExampleXiSchreier} and Example \ref{ExampleXiSchreierStar} that the inequalities of Theorem \ref{TheoremQuantBS} are optimal and the second and third inequalities may be strict. We note that in both of these examples we have $\operatorname{wbs}_\xi (A) = \operatorname{wbs}_\xi^s(A) = 2\operatorname{sm}_{\xi+1}(A)$. We do not know if these inequalities can be strict.

In \cite{Bendov__2015} the authors asked, whether for a bounded set $A$ in a Banach space $X$ it is necessarily true that
\begin{align*}
    \operatorname{wbs}(A) = 2\operatorname{sm}(A) = 2\operatorname{wus}(A).
\end{align*}
(For the definition of these quantities see \cite{Bendov__2015}, note that $\operatorname{wbs}(A) = \operatorname{wbs}_0(A)$, $\operatorname{sm}(A) = \operatorname{sm}_1(A)$ and $\operatorname{wus}(A) = \operatorname{wus}_1(A)$ in our notation). Example \ref{ExampleXiSchreierStar} answers this question negatively.

In the next example we will demonstrate the need of separability in Proposition \ref{PropDeltaSeparable}. Our non-separable space will be the $\ell_2$-sum of the Schreier-Baernstein spaces, which are, in a way, reflexive versions of the Schreier spaces defined in Example \ref{ExampleXiSchreier}.

\begin{example} \label{ExampleDeltaNonseparable}
    There is a non-separable reflexive Banach space $X$ with $\delta_0(B_X) = 2$. That is, $\delta_0$ is not a measure of weak non-compactness on $X$.
\end{example}

\begin{proof}
    For $\xi < \omega_1$ let us consider the Schreier-Baernstein space $X_\xi^2$, that is the completion of $c_{00}$ under the norm
    \begin{align*}
        \norm{x}_{X_\xi^2} = \sup \left\{ \left( \sum_{j=1}^n (\sum_{i \in F_j} |x_i|)^2 \right)^{\frac{1}{2}}: \; F_1 < F_2 < \dots < F_n \in \mathcal{S}_\xi \right\}.
    \end{align*}
    Then the canonical sequence $(e_n)_{n \in \mathbb{N}}$ of $c_{00}$ is a shrinking boundedly-complete basis of $X_\xi^2$, see \cite[Lemma 3.2.]{Causey2015}. In particular, $(e_n)_{n \in \mathbb{N}}$ is weakly null. It also immediately follows from the definition of the norm $\norm{\cdot}_{X_\xi^2}$ that $(e_n)_{n \in \mathbb{N}}$ generates an $\ell_1^\xi$-spreading model with constant $1$.
    
    Let us now consider the $\ell_2$-sum of the spaces $X_\xi^2$,
    \begin{align*}
        X = \ell_2-\bigoplus_{\xi < \omega_1} X_\xi^2.
    \end{align*}
    Then $X$ is a non-separable reflexive Banach space, as the spaces $X_\xi^2$ are reflexive by the result of James, see e.g. \cite[Theorem 3.2.13.]{nigel2006topics}. It follows that $B_X$ is weakly compact. But $\operatorname{sm}_\xi(B_X) \geq 1$ for all $\xi < \omega_1$, as $B_X$ contains isometric copies of the canonical bases of the spaces $X_\xi^2$. It follows from Theorem \ref{TheoremQuantBS} and Proposition \ref{PropCharBS} that $\operatorname{bs}^s_\xi(B_X) \geq 2$. The other inequality is trivial, hence, $\operatorname{bs}^s_\xi(B_X) = 2$ for all $\xi < \omega_1$, and thus $\delta_0(B_X) = 2$.
\end{proof}

\section{Remarks and open problems} \label{section:Remarks}

First, let us show that the quantities $\operatorname{sm}_\xi$ and $\operatorname{wus}_\xi$ do not depend on the choice of successor ordinals made in the definition of the Schreier hierarchy.

\begin{lemma} \label{LemmaDepend}
    Let $(\mathcal{S}_\xi)_{\xi < \omega_1}$ and $(\mathcal{G}_\xi)_{\xi < \omega_1}$ be two Schreier hierarchies with potentially different choices of sequences of successor ordinals defining the families $\mathcal{S}_\xi$ and $\mathcal{G}_\xi$ for limit ordinals $\xi$. Let $(x_n)_{n \in \mathbb{N}}$ be a weakly null sequence in a Banach space $X$ and $c > 0$.
    \begin{itemize}
        \item If $(x_n)_{n \in \mathbb{N}}$ generates an $\ell_1^\xi$-spreading model with respect to $\mathcal{S}_\xi$ and with constant $c$, then there is $M \in [\mathbb{N}]$ such that $(x_n)_{n \in M}$ generates an $\ell_1^\xi$-spreading model with respect to $\mathcal{G}_\xi$ and with constant $c$.
        
        \item If $(x_n)_{n \in \mathbb{N}}$ in $(\xi,c)$-large with respect to $\mathcal{S}_\xi$, then there is $N \in [\mathbb{N}]$ such that $(x_n)_{n \in N}$ is $(\xi,c)$-large with respect to $\mathcal{G}_\xi$.
    \end{itemize}
\end{lemma}

\begin{proof}
    It follows from \cite[Theorem 2.2.6.]{ArgyrosMercourakusTsarpalias} that there is $M = (m_k)_{k \in \mathbb{N}} \in [\mathbb{N}]$ such that $\mathcal{G}_\xi^M \subseteq \mathcal{S}_\xi$. For the first part, we want to show that
    \begin{align*}
        \norm{\sum_{k \in F} a_k x_{m_k}} \geq c \sum_{k \in F} |a_k| \hspace{1cm} \text{for all } F \in \mathcal{G}_\xi \text{ and } (a_k)_{k \in F} \in \mathbb{R}^F.
    \end{align*}
    Fix such $F$ and $(a_k)_{k \in F}$ and define $b_j = a_k$ if $j = m_k$ for some $k \in F$ and $b_j = 0$ otherwise. Then $F' = \{m_k: \; k \in F\} \in \mathcal{G}_\xi^M \subseteq \mathcal{S}_\xi$ and
    \begin{align*}
        \sum_{k \in F} a_k x_{m_k} = \sum_{k \in F} b_{m_k} x_{m_k} = \sum_{j \in F'} b_j x_j.
    \end{align*}
    Hence,
    \begin{align*}
        \norm{\sum_{k \in F} a_k x_{m_k}} = \norm{\sum_{j \in F'} b_j x_j} \geq c \sum_{j \in F'} |b_j| = c \sum_{k \in F} |a_k|.
    \end{align*}
    
    The second part is easier -- if there is $N \in [\mathbb{N}]$ such that $\mathcal{S}_\xi^N \subseteq \mathcal{F}_c((x_n)_{n \in \mathbb{N}})$, then $\mathcal{S}_\xi \subseteq \mathcal{F}_c((x_n)_{n \in N})$. Hence, $\mathcal{G}_\xi^M \subseteq \mathcal{S}_\xi \subseteq \mathcal{F}_c((x_n)_{n \in N})$ and $(x_n)_{n \in N}$ is $(\xi,c)$-large with respect to $\mathcal{G}_\xi$.
\end{proof}

It easily follows from the previous lemma that the quantities $\operatorname{sm}_\xi$ and $\operatorname{wus}_\xi$ do not depend on the choice of successor ordinals made in definition of the Schreier hierarchy. We do not know if the quantities $\operatorname{wbs}_\xi$ and $\operatorname{wbs}_\xi^s$ depend on this choice, however, by Theorem \ref{TheoremQuantBS}, they are equivalent to the quantity $\operatorname{sm}_{\xi+1}$, which is independent on this choice. Hence, the notions of weak $\xi$-Banach-Saks sets are also not dependent on this choice.

As we already mentioned in Section \ref{section:Examples}, the inequalities of Theorem \ref{TheoremBSaWBS} are optimal and, possibly except for the inequality $\operatorname{bs}_\xi(A) \leq \operatorname{bs}_\xi^s(A)$, can be strict. We have also shown that the inequalities of Theorem \ref{TheoremQuantBS}. are optimal and the inequalities concerning the quantity $\operatorname{wus}_{\xi+1}$ can be strict. What remains open is the following question:

\begin{question} \label{QuestionStrict}
    Let $A$ be a bounded set in a Banach space $X$ and $\xi < \omega_1$. It is necessarily true that $\operatorname{wbs}_\xi(A) = \operatorname{wbs}_\xi^s(A) =  2\operatorname{sm}_{\xi+1}(A)$?
\end{question}

It follows from Theorem \ref{TheoremQuantBS} that the quantities $\operatorname{wbs}_\xi$ and $\operatorname{wbs}_\xi^s$ are equivalent. The same approach, however, cannot be used for the quantities $\operatorname{bs}_\xi$ and $\operatorname{bs}_\xi^s$.

\begin{question}
    Are the quantities $\operatorname{bs}_\xi$ and $\operatorname{bs}_\xi^s$ equal? Or, at least, equivalent?
\end{question}

In \cite[Section 5]{Bendov__2015} the authors proved a dichotomy concerning the quantities applied to a unit ball. More precisely, they showed, in our notation, that for a Banach space $X$ we have $\operatorname{wbs}_0(B_X) \in \{0,2\}$. We did not manage to use this approach to the quantities of higher orders, so the following question still remains open:

\begin{question} \label{QuestionDichotomy}
    Let $X$ be a Banach space and $\xi<\omega_1$. Is it necessarily true that $\operatorname{wbs}_\xi(B_X) \in \{0,2\}$?
\end{question}

It is known that a normalised basic sequence $(x_n)_{n \in \mathbb{N}}$ in a Banach space $X$ has a subsequence generating a spreading model, say $\mathcal{X}$ (see e.g. \cite[Theorem 11.3.7.]{nigel2006topics}). It is readily proved that if moreover $(x_n)_{n \in \mathbb{N}}$ generates an $\ell_1$-spreading model, then $\mathcal{X}$ is isomorphic to $\ell_1$. This in combination with a variation of the James' $\ell_1$ distorsion theorem \cite[Theorem 10.3.1.]{nigel2006topics} was used in \cite{Bendov__2015} to prove the dichotomy for $\xi = 0$. It could help to solve Question \ref{QuestionDichotomy} if we could say something more about the relation of $(x_n)_{n \in \mathbb{N}}$ and $\mathcal{X}$ if we knew that $(x_n)_{n \in \mathbb{N}}$ generates an $\ell_1^\xi$-spreading model for some $1 < \xi < \omega_1$.

\bibliographystyle{acm}
\bibliography{bibliography}

\end{document}